\documentclass[12pt,reqno]{amsart} 
\usepackage{amssymb}
\usepackage{amsthm}
\usepackage{amsmath}
\usepackage{verbatim}
\usepackage{enumerate}
\usepackage[usenames,dvipsnames] {color}
\numberwithin{equation}{section}
\newcommand\de{\delta}
\newcommand\lam{\lambda}
\renewcommand\phi{\varphi}
\def\d{\mathbb{D}}
\def\c{\mathbb{C}}
\newcommand\eps{\varepsilon}
\newcommand\g{\mathcal{G}}

\def\r{\mathbb{R}}

\newcommand\ov{\overline}
\newcommand\half{\tfrac 12}
\def\hol{{\rm Hol}}
\newcommand\inv{^{-1}}

\newcommand\bbm{\begin{bmatrix}}
\newcommand\ebm{\end{bmatrix}}
\newcommand\bpm{\begin{pmatrix}}
\newcommand\epm{\end{pmatrix}}

\def\h{\mathcal{H}}

\def\b{\mathcal{B}}
\def\m{\mathcal{M}}
\def\n{\mathcal{N}}

\def\ph{\phi}
\newcommand\ga{\gamma}
\def\f{\mathcal{F}}

\newcommand\calk{\mathcal{K}}

\newcommand\df{\stackrel{\rm def}{=}}
\newcommand\ess{\mathcal{S}}

\newcommand\x{\mathcal{X}}
\newcommand\y{\mathcal{Y}}
\def\be{\begin{equation}}
\def\ee{\end{equation}}

\def\hol{{\rm Hol}}
\def\hinf{{\rm H}^\infty}

\def\ess{\mathcal{S}}
\def\essdp{\mathcal{S}_{\mathrm {dp}}}

\def\fdp{\mathcal{F}_{\rm dp}}
\def\gdp{\mathcal{G}_{\rm dp}}

\def\ip#1#2{\langle\, #1\, ,#2\, \rangle}

\def\Bigip#1#2{\Big\langle\, #1\, ,#2\, \Big\rangle}

\def\b{\mathcal{B}}

\def\hol{{\rm Hol}}
\def\hinf{{\rm H}^\infty}

\def\set#1#2{\{ #1 \, : \, #2\}}

\def\norm#1{\| #1 \|}

\def\ip#1#2{\langle #1,#2\rangle}

\def\s0{s_0}
\def\p0{p_0}
\DeclareMathOperator{\ran}{ran}

\DeclareMathOperator{\rank}{rank}

\theoremstyle{definition}
\newtheorem{defin}[equation]{Definition}
\newtheorem{lem}[equation]{Lemma}
\newtheorem{prop}[equation]{Proposition}

\newtheorem{rem}[equation]{Remark}
\newtheorem{thm}[equation]{Theorem}

\newtheorem{prob}[equation]{Problem}

\newtheorem{ex}[equation]{Example}

\newcommand\black{\color{black}}

\begin{document} 
\title[Function theory on the annulus  in the dp-norm]{Function theory on the annulus in the dp-norm}

\author{Jim Agler}
\address{Department of Mathematics, University of California at San Diego, CA \textup{92103}, USA}
\email{jagler@ucsd.edu}

\author{Zinaida A. Lykova}
\address{School of Mathematics,  Statistics and Physics, Newcastle University, Newcastle upon Tyne
	NE\textup{1} \textup{7}RU, U.K.}
\email{Zinaida.Lykova@ncl.ac.uk}

\author{N. J. Young}
\address{School of Mathematics, Statistics and Physics, Newcastle University, Newcastle upon Tyne NE1 7RU, U.K.}
\email{Nicholas.Young@ncl.ac.uk}

\date{28th September 2025}

\keywords{Holomorphic functions; Hilbert space model; annulus; Crouzeix conjecture; Douglas-Paulsen family}

\dedicatory{ In Memory of Rien Kaashoek}

\subjclass{ 47B99, 30E05, 32A26}


\thanks{Partially supported by National Science Foundation Grants
	DMS 1361720 and 1665260, the Engineering and Physical Sciences Research Council grant EP/N03242X/1.}

\begin{abstract} 
In this paper we shall use realization theory, a favourite technique of Rien Kaashoek, to prove new results about a class of holomorphic functions on an annulus
 \[
R_\delta \df \{z\in\c: \delta <|z|<1\},
\]
where $0<\delta<1$.
The class of functions in question arises in the early work of R. G. Douglas and V. I. Paulsen on the rational dilation of a Hilbert space operator $T$ to a normal operator with spectrum in $\partial R_\delta$.
Their work suggested the following norm $\|\cdot\|_{\mathrm{dp}}$ on the space 
$\mathrm{Hol}(R_\delta)$ of holomorphic functions on $R_\delta$, 
\[
\|\ph\|_{\mathrm{dp}} \stackrel{\rm def}{=}  \sup\{ \|\ph(T)\|:  \|T\|\leq 1,  \|T\inv\|\leq 1/\delta \ \text{and} \   \sigma(T)\subseteq R_\delta\}.
\]
By analogy with the classical Schur class of holomorphic functions  $\mathcal{S} $ with supremum norm at most $1$ on the disc $\d$, it is natural to consider the {\em dp-Schur class} $\mathcal{S}_\mathrm{dp}$ of holomorphic functions of dp-norm at most $1$ on $R_\delta$.

Our central result  is a Pick interpolation theorem for functions in $\mathcal{S}_\mathrm{dp}$ that is analogous to Abrahamse's Interpolation Theorem for bounded holomorphic functions on a multiply-connected domain.
For a tuple $\lambda=(\lambda_1,\dots,\lambda_n)$ of distinct interpolation nodes in $R_\delta$, we introduce a special set $\g_\mathrm{dp}(\lambda)$ of positive definite $n\times n$ matrices, which we call {\em DP Szeg\H{o} kernels}.
The DP Pick problem $\lam_j \mapsto z_j, j=1,\dots,n$, is shown to be solvable if and only if,
\[ 
[(1-\overline z_i z_j)g_{ij}] \ge 0 \; \text{ for all}\;
g \in \mathcal{G}_{\mathrm {dp}} (\lambda).
\]
We prove further that a solvable DP Pick problem has a solution which is a rational function with a finite-dimensional model, an intriguing result which opens up the possibility of a theory of extremal functions from $\mathcal{S}_\mathrm{dp}$ analogous to the theory of finite Blaschke products.
\end{abstract} 

\maketitle
\tableofcontents

\section{Introduction}

It is our honour to contribute to this memorial issue for Marinus Kaashoek, who was a prolific and influential operator theorist throughout a long career.
A constant thread in his research over several decades was the power of realization theory applied to a wide variety of problems in analysis.   Among his many contributions in this area we mention his monograph \cite{bgk}, written with his longstanding collaborators Israel Gohberg and Harm Bart, which was an early and influential work in the area, and his more recent papers and book, including \cite{KaavS2014,FHK2014, KaaVL}.  Realization theory uses explicit formulae for functions in terms of operators on Hilbert space to prove function-theoretic results. 
In this paper we continue along the Bart-Gohberg-Kaashoek path by exploiting realization theory to prove new results about a class of holomorphic functions which was first encountered by R. G. Douglas and V. I. Paulsen in a study of rational dilation on the annulus.

For any open set  $\Omega$ in the plane, $\hol(\Omega)$ will denote the set of holomorphic functions  on $\Omega$ and $\hinf(\Omega)$ will denote the Banach algebra of bounded holomorphic functions on $\Omega$, equipped with the supremum norm $\norm{\phi}_{\hinf(\Omega)}= \sup_{z\in\Omega}|\phi(z)|$.   Let $\ess(\Omega)$ denote the class $ \{ \phi \in \hinf(\Omega): \|\phi\|_{\hinf(\Omega)} \leq 1 \}$.
The classical Schur class, $\ess$, is the set $\ess(\d)$.

We recall the extensively-studied Pick interpolation theorem \cite{pick} for bounded holomorphic functions on the open unit disc $\d$.  
\begin{thm} \label{pick-intro} Let  $\lam_1,\dots,\lam_n \in \d$ be distinct  and let $z_1,\dots,z_n\in\c$.
There exists $ \phi \in \ess$ such that
\[\phi(\lambda_j) = z_j \qquad \text{for} \ j=1,\ldots,n, \]
if and only if,
\[
 \bbm \displaystyle \frac{1-\ov {z_i} z_j}{1-\ov{\lam_i}\lam_j} \ebm_{i,j=1}^n \geq 0.
\]
\end{thm}
Pick interpolation problems, with the unit disc replaced by other domains $\Omega$ in the plane, have also been much studied.
 In the event that $\Omega$ is a simply connected proper open subset of $\c$, with the aid of the conformal map $F:\Omega \to \d$, we can convert this problem into a classical Pick problem on $\d$ with interpolation data $F(\lam_j) \mapsto w_j$ for $ j=1,\ldots,n,$ and then Pick's theorem gives a criterion for the existence of $\phi$ in terms of the positivity of the appropriate ``Pick matrix", which here is
\[
 \bbm \displaystyle \frac{1-\ov {w_i} w_j}{1-\ov{F(\lam_i)}F(\lam_j)} \ebm_{i,j=1}^n \geq 0.
\]

 More generally, the Pick problem on a multiply connected domain was studied in the 1940s by Garabedian \cite{gar} and Heins \cite{heins}. Later, Sarason \cite{sar2} and Abrahamse \cite{Abr1977} treated the problem  in terms of reproducing kernels, an approach that we follow in this paper.  Abrahamse's Theorem gives a solution to the Pick interpolation problem on any bounded domain  $\Omega$  in the plane whose boundary consists of finitely many disjoint analytic Jordan curves. He showed that a Pick problem  on $\Omega$  can be solved if and only if an infinite collection of Pick matrices are positive semi-definite.
In the case of the annulus $ R_\de=\{z\in\c:\de < |z| < 1\}$, for a tuple $\lam=(\lam_1,\dots,\lam_n)$ of distinct interpolation nodes in $ R_\de$,
Abrahamse \cite{Abr1977} described a family $\g(\lam)$ of positive definite   $n \times n$ matrices
for which the following statement is true: \black
\begin{thm}\label{Abrahamse}  Let  $\lam_1,\dots,\lam_n \in R_\de$ be distinct  and let $z_1,\dots,z_n\in\c$.
There exists $ \phi \in \ess(R_\de)$ such that
\[\phi(\lambda_j) = z_j \qquad \text{for} \ j=1,\ldots,n, \]
if and only if, for each $g\in \g(\lam)$,
\[
 \bbm (1-\ov{z_i}z_j) g_{ij} \ebm_{i,j=1}^n \ge 0.
\]
\end{thm}

An alternative  explicit choice of $\g(\lam)$ for which Theorem \ref{Abrahamse} is true  is described in \cite{sar2, Ag2} as follows
$$\g(\lam)= \{ \bbm  g_\rho(\lam_i,\lam_j)\ebm_{i,j=1}^n: \rho > 0\},$$
where
\[
g_\rho(\lam_i,\lam_j) =  \sum_{m=-\infty}^\infty \frac {(\overline\lam_i \lam_j)^m}{\rho + \de^{2m}},\quad \text{for}\ 1\leq i,j \leq n.
\]

Another natural variant of Pick's problem arises if one replaces the supremum norm on $\hol(\Omega)$ by a different norm. For example, consider the {\em Dirichlet space} $\mathcal{D}$ of holomorphic functions $f$ on $\d$ such that $f'$ is square integrable with respect to area measure on $\d$, with pointwise operations and the norm
\[
\|f\|_\mathcal{D}^2 = \sum_{n=0}^\infty (n+1) |\hat{f}(n)|^2 = \|f\|_{H^2}^2 + \int_\d|f'(z)|^2 dm(z),
\]
where $m$ denotes area measure on the disc.
The Dirichlet space is a Hilbert function space on $\d$ with reproducing kernel
\[
k_{\mathcal{D}}(\lam,\mu)= -\frac{1}{\overline\mu \lam}\log(1-\overline\mu\lam).
\]
The Pick-type interpolation problem appropriate to this Hilbert function space is expressed in terms of its {\em multiplier space} $\m(\mathcal{D})$, which is defined to be the space of functions $\phi$ on $\d$ such that $\phi f \in\mathcal{D}$ for every $f\in\mathcal{D}$, with pointwise operations and the {\em multiplier norm} 
\[
\|\ph\|_{\m(\mathcal{D})} = \sup \{\|\ph f\|_\mathcal{D}: f\in\mathcal{D}, \|f\|_{\mathcal{D}} \leq 1\}.
\]
In this setting the corresponding Pick interpolation theorem is the following \cite[Corollary 7.41]{AgMcC2002}: 
\begin{thm}\label{Dirichlet}  Let  $\lam_1,\dots,\lam_n \in \d$ be distinct  and let $z_1,\dots,z_n\in\c$.
There exists $\ph\in\mathcal{M}(\mathcal{D})$ such that $\|\ph\|_{\mathcal{M}(\mathcal{D})} \leq 1$ and 
\[\phi(\lambda_j) = z_j \qquad \text{for} \ j=1,\ldots,n, \]
if and only if
\[
 \bbm (1- z_i \ov{z_j}) k_\mathcal{D}(\lam_i,\lam_j) \ebm_{i,j=1}^n \ge 0.
\]
\end{thm}
An account of Pick theorems in the context of sundry different Hilbert function spaces, including $\mathcal{D},$ may be found in the book \cite{AgMcC2002}.

In this paper we will deviate from the supremum norm on $\hol(R_\de), \de\in(0,1)$.
An operator $X$ on a Hilbert space is called a {\em Douglas-Paulsen operator with parameter $\de$} if $\|X\|\leq 1$ and  $\|X\inv\|\leq 1/\de$, see  \cite{dp86}.
The {\em Douglas-Paulsen family}, $\f_{\mathrm{dp}}(\de)$, is the class of Douglas-Paulsen operators $X$ with parameter $\de$ such that  $\sigma(X)\subseteq R_\de$. 
We consider the {\em Douglas-Paulsen norm}\footnote{$\|\cdot\|_{\mathrm{dp}}$ is an example of a {\em calcular norm}, see \cite[Chapter 9]{amy20}}
\be \label{dpnorm}
\norm{\phi}_{\mathrm {dp}} = \sup_{X\in  \f_{\mathrm{dp}}(\de) }\norm{\phi(X)},
\ee
defined for $\phi \in \hol(R_\delta)$.
There is no guarantee that the quantity defined by equation \eqref{dpnorm} is finite.
Accordingly, we introduce the associated Banach algebra
\[
\hinf_{\mathrm {dp}}(R_\de)=\set{\phi \in \hol(R_\delta)}{\norm{\phi}_{\mathrm {dp}}<\infty}.
\]
In addition, we introduce the {\em dp-Schur class}, $\essdp$\footnote{In the notations $\|\cdot\|_{\mathrm {dp}}$ and $\mathcal{S}_{\mathrm {dp}}$ we suppress dependence on the parameter $\de$.},  which is the set of functions $\ph\in \hol(R_\de)$ such that
$\|\ph\|_{\mathrm {dp} }\leq 1.$

 An important step in the Douglas-Paulsen theory was the following estimate.
 If $X$ is a Douglas-Paulsen operator with parameter $\de$, $\sigma(X)\subseteq R_\de$ and $\phi$ is a bounded holomorphic matrix-valued function on $R_\de$ then
\be\label{dps}
\|\phi(X)\| \leq \left(2+\frac{1+\de}{1-\de}   \right) \sup_{z\in R_\de} \|\phi(z)\|.
\ee

Hence, we see from equations \eqref{dpnorm} and \eqref{dps} that 
$$\|\ph\|_\mathrm{dp} \leq \left(2+\frac{1+\de}{1-\de}   \right) \|\ph\|_{\hinf(R_\de)}$$ for
 $\ph\in\hol(R_\de)$. On the other hand, see Remark \ref{dp-inf}, 
 $\|\ph\|_{\hinf(R_\de)} \leq 
 \|\ph\|_\mathrm{dp} $, and so
the dp and supremum norms on $\hol(R_\de)$ are equivalent.  Thus,
\[
\mathrm{H}^\infty(R_\de) = \mathrm{H}_\mathrm{dp}^\infty(R_\de)
\]
as sets.  However, the reader should be aware that
\[\|\cdot\|_\mathrm{dp} \neq \|\cdot\|_{\hinf(R_\de)} \ \text{and therefore} \ \essdp \neq \ess(R_\de),
\]
a fact that Example \ref{dp-neq-infty} below demonstrates.

The power of inequality  \eqref{dps} is that it holds for all {\em matrix-valued} functions $\ph$, a fact which allowed Douglas and Paulsen to show  that if $T \in \b(\h)$ is a Douglas-Paulsen operator, then there exists an invertible $S \in \b(\h)$ such that 
\be
\|S\| \|S^{-1}\|\leq 2+\frac{1+\de}{1-\de}  
\ee
and $S T S^{-1}$ dilates  to a normal operator  with spectrum contained in the boundary $\partial R_\de$.
This result is a kind of Nagy dilation theorem for the annulus. 
In the {\em scalar case} a slightly stronger result than the inequality \eqref{dps} had been obtained earlier by A. Shields \cite[Proposition 23]{Shields}, with the smaller constant $2+\sqrt{\frac{1+\de}{1-\de}}$ on the right hand side.
Shields asked whether the constant $2+\sqrt{\frac{1+\de}{1-\de}}$  could be replaced by a quantity that remains bounded as $\de\to 1$.  This question was answered 
in the affirmative 
 by  C. Badea, B. Beckermann and M. Crouzeix  \cite{bbc2009} 
and subsequently the better constant $1+\sqrt{2}$ was established by
M. Crouzeix and A. Greenbaum \cite{Crou_Greenb}.\\

Corresponding to the dp-Schur class there is a natural variant of the classical Pick interpolation problem, which we call the {\em DP Pick problem}: 
given $n$ distinct points $\lambda_1,\ldots,\lambda_n$ in $R_\delta$ and $z_1,\ldots,z_n \in \c$, does there exist a function $\phi \in \hinf_{\mathrm {dp}}(R_\de)$ 
with $\norm{\phi}_{\mathrm {dp}} \le 1$ such that 
\be\label{200int}
\phi(\lambda_j) = z_j \qquad \text{for} \ j=1,\ldots,n ?
\ee

We shall show that there is a solvability criterion for this problem which is  parallel to Abrahamse's Theorem, but with $\g(\lam)$ replaced by a collection $\g_\mathrm{dp}(\lam)$ of kernels, which we now define.

\begin{defin}\label{dp-kernels-intro}
Let $\lambda_1, \dots, \lambda_n$ be $n$ distinct points in $R_\de$ and let $\lam=(\lam_1,\dots,\lam_n)$.
A  {\em DP Szeg\H{o} kernel} for the $n$-tuple $\lam$ is a  positive definite  $n \times n$ matrix $g = [g_{ij}]$ such that 
\be\label{210int}
[(1-\overline\lambda_i \lambda_j)g_{ij}] \ge 0\ \ \text{ and }\ \ 
\left[\left(1-\overline{\frac{\delta}{\lambda_i}} \frac{\delta}{\lambda_j}\right)g_{ij}\right] \ge 0.
\ee
The set of all DP Szeg\H{o} kernels for the $n$-tuple $\lam$ will be denoted by 
$\g_\mathrm{dp}(\lam)$.

\end{defin}

We observe that $\g_\mathrm{dp}(\lam)$
 consists of the gramians $[\ip{e_j}{e_i}]_{i,j=1}^n$ for all bases $e_1,\dots,e_n$ of an $n$-dimensional Hilbert space $\h$ such that the operator $T$ on $\h$ defined by $Te_j= \lam_je_j$ for $j=1,\dots,n$ is a Douglas-Paulsen operator.
 This and related facts are described in Section \ref{dpk}.

 The Pick interpolation theorem for the dp-norm on $\hol(R_\de)$ is the following statement (which is Theorem \ref{iff-dp-kernel-z}  from the body of the paper).

\begin{thm} \label{iff-dp-kernel-intro} Let  $\lam_1,\dots,\lam_n \in R_\de$ be distinct  and let $z_1,\dots,z_n\in\c$.
There exists $ \phi \in \essdp$ such that
\[\phi(\lambda_j) = z_j \qquad \text{for} \ j=1,\ldots,n, \]
if and only if, for all 
$g \in \mathcal{G}_{\mathrm {dp}} (\lambda)$,
\be\label{222-intro}
[(1-\overline z_i z_j)g_{ij}] \ge 0.
\ee
\end{thm}

In Section \ref{dp-sup-section} we compare the dp norm and the sup norm of a function in $\hol(R_\de)$ and we point out a connection to the Crouzeix conjecture.  In Section \ref{model-dp} we review the theory of models and realizations of holomorphic functions on $R_\de$ with dp-norm at most $1$, see Theorem \ref{thm3.8}.  In Section \ref{dpk} we introduce 
 DP-Szeg\H{o} kernels on an $n$-tuple of points in $R_\de$ and elaborate their relation to the Douglas-Paulsen class.
In Section \ref{dp-Pick} we recall another approach to the solution of DP Pick problems given in \cite[Theorem 9.46]{amy20}, and we show that solvable DP Pick problems have {\em rational} solutions. 
In Section \ref{extremal} we consider an extremally solvable DP Pick problem
$\lam_j \mapsto z_j$ for $j=1,\dots,n$, 
and show that, for such a problem there is a rational solution $\ph\in\essdp$ and 
there exists a Douglas-Paulsen operator $T$ with parameter $\de$  which acts on an $n$-dimensional Hilbert space with $\sigma(T)=\{\lam_1,\dots,\lam_n\}$ such that
$
\norm{\phi}_{\rm dp} = \norm{\phi(T)}=1,
$
see Theorem \ref{finite-extr}.

\section{The dp and sup norms on $\hol(\d)$ and $\hol(R_\delta)$} \label{dp-sup-section}

In this section we describe connections between
the Banach algebra
$\hinf_{\mathrm {dp}}(R_\de)$
and the Crouzeix  conjecture. We will prove in Proposition \ref{phi-in-holD} 
that there is a large class of functions
 $\phi \in \hol(R_\de)$, such that 
\[
\norm{\phi}_{\mathrm {dp}} = \|\phi\|_{\hinf (R_\de)}.
\]
In Example \ref{dp-neq-infty} below we show that the last relation fails to hold for the function $\ph\in\hol (R_\de)$ defined by $ \ph(z)=z +\frac{\delta}{z}$ for $z\in R_\de$.  In fact $\ph$ satisfies
\[
\| \ph \|_{\mathrm{dp}} = 2 \ \text{and}\ \| \ph\|_{\hinf (R_\de)}  = 1+\de.
\]
By an {\em elliptical domain} we shall mean the domain in the complex plane bounded by an ellipse. As a standard elliptical domain we take the set 
\be\label{defGde}
G_\de  \df \{x+iy:x,y\in\r, \frac{x^2}{(1+\de)^2} + \frac{y^2}{(1-\de)^2} < 1\},
\ee
for some $\de$ such that $0\leq\de<1$.
Note that any elliptical domain can be identified via an affine self-map of the plane with an elliptical domain of the form $G_\de$ for some $\de\in[0,1)$.

 In this paper all Hilbert spaces are complex Hilbert spaces.
For a complex Hilbert space $\h$ we denote by $\b(\h)$ the space of bounded operators on $\h$. 
If  $T\in\b(\h)$, then $W(T)$, \emph{the numerical range of $T$}, is defined by the formula
\[
W(T) = \set{\ip{Tx}{x}_\h}{x\in\h, \norm{x}=1 }.
\]

The {\em B. and F. Delyon family}, $\f_{\mathrm {bfd}}(C)$, corresponding to an open bounded convex set $C$ in $\c$ is the class of operators $T$ such that the closure of the numerical range of $T$, 
   $\overline{W(T)}$, is contained in $C$.  By \cite[Theorem 1.2-1]{GuRao97}, the spectrum $\sigma(T)$  of an operator $T$ is contained in $\overline{W(T)}$, and 
   so, by the Riesz-Dunford functional calculus, $\ph(T)$ is defined for all $\phi\in \hol(C)$ and $T\in \f_{\mathrm {bfd}}(C)$. Therefore, we may consider the calcular norm\footnote{A {\em calcular norm} on a function space is a norm that is defined with the aid of the functional calculus.  For more information on such norms the reader may consult \cite[Chapter 9]{amy20}.} 
\be \label{bfdnorm}
\norm{\phi}_{\f_{\mathrm {bfd}}(C)}=\sup_{T\in \f_{\mathrm {bfd}}(C)}\norm{\phi(T)},
\ee
defined for $\phi \in \hol(C)$, and 
the associated Banach algebra
\[
\hinf_{\mathrm {bfd}}(C)=\set{\phi \in \hol(C)}{
\norm{\phi}_{\f_{\mathrm {bfd}}(C)} <\infty}.
\]
In this paper the convex set $C$ will always be $G_\de$, and so we abbreviate the notation to $\|\cdot\|_{\mathrm {bfd}}$ in place of 
$\|\cdot\|_{\f_{\mathrm {bfd}}(G_\de)}$. Thus
\be\label{bfdnorm2}
\norm{\phi}_{\mathrm {bfd}}=\sup_{T\in \f_{\mathrm {bfd}}(G_\de)}\norm{\phi(T)},
\ee
defined for $\phi \in \hol(G_\delta)$.
In addition we introduce the {\em bfd-Schur class}, $\mathcal{S}_{\mathrm {bfd}}$, of functions on $G_\de$, which is the set of functions
$f \in \hol(G_\de)$ such that
$\|f\|_{\mathrm {bfd} }\leq 1.$\footnote{In the notations $\|\cdot\|_{\mathrm {bfd}}$ and $\mathcal{S}_{\mathrm {bfd}}$ we suppress dependence on the parameter $\de$.}
The bfd-norm is named in recognition of a celebrated theorem \cite{bfd1999} of the brothers B. and F. Delyon, 
which states that, if $p$ is a polynomial, $\h$ is a Hilbert space and $T\in \b(\h)$ then 
\[
\|p(T)\| \leq \kappa(W(T))\|p\|_{W(T)},
\]
where  $\|\cdot\|_{W(T)}$ denotes the supremum norm on $W(T)$, and, for any bounded convex set $C$ in $\c$, $\kappa(C)$ is defined by
\[
\kappa(C) = 3+\left(\frac{2\pi(\mathrm{diam}(C))^2}{\mathrm{area}(C)}\right)^3.
\]
Let us write
\[
K(\f_{\mathrm {bfd}}(C)) = \sup_{\phi \in \hol(C): \|\phi\|_{\hinf(C)} \le 1} \norm{\phi}_{\f_{\mathrm {bfd}}(C)},
\]
and the Crouzeix universal constant
\[
K_{\mathrm {bfd}} =\sup \{ K(\f_{\mathrm {bfd}}(C)): \ C \ \text{is a  bounded convex set in } \ \c\}.
\]
In \cite{Cr2004}, Crouzeix proved $K_{\mathrm {bfd}} \le 12$ and conjectured that $ K_{\mathrm {bfd}} = 2$.
Subsequently Crouzeix and Palencia \cite{CrPal2017} proved that $ K_{\mathrm {bfd}} \leq 1+\sqrt{2}$. Still more recently  Crouzeix and Kressner \cite{Crou_Kress} showed that $W(T)$ is a {\em complete} $(1+\sqrt{2})$-spectral set for $T$.

Let $\pi: R_\de \to G_\de$ be defined by $\pi(z)=z +\frac{\delta}{z}$, $ z \in R_\de$.
Now observe that if $\phi \in \hol(G_\delta)$
 then we may define $\pi^\sharp(\phi) \in \hol(R_\delta)$ by the formula
\[
\pi^\sharp(\phi)(\lambda)=\phi(\pi(\lambda)) \qquad\text{ for all } \lambda\in R_\delta.
\] 
We record the following simple fact from complex analysis without proof.
\begin{lem}\label{dp.lem.10} Let $\de \in (0,1)$ and let
 $\psi \in \hol(R_\delta)$. Then $\psi \in \ran \pi^\sharp$ if and only if $\psi$ is \emph{symmetric} with respect to the involution $\lam\mapsto\de/\lam$ of $R_\de$, that is, if and only if $\psi$ satisfies
\[
\psi(\delta/\lambda)=\psi(\lambda)
\]
for all $\lambda \in R_\delta$.
\end{lem}
The following result, which is \cite[Theorem 11.25]{aly23},  gives an intimate connection between the $\norm{\cdot}_{\mathrm {dp}}$ and $\norm{\cdot}_{\mathrm {bfd}}$ norms.
\begin{thm}\label{dp.thm.10}  Let $\de \in (0,1)$.
The mapping $\pi^\sharp$ is an isometric isomorphism from $\hinf_{\mathrm {bfd}}(G_\de)$ onto the set of symmetric functions with respect to the involution $\lam\mapsto\de/\lam$ in $\hinf_{\mathrm {dp}}(R_\de)$,  so that, for all $\phi\in \hol(G_\de)$,
\be\label{bfd=dp}
\|\phi\|_{\mathrm {bfd}} = \|\phi\circ\pi\|_{\mathrm {dp}}.
\ee
\end{thm}

\begin{rem}\label{dp-inf}
One can see that, for $\phi \in \hol(R_\delta)$, 
\begin{align} \label{dp-sup}
\norm{\phi}_{\mathrm {dp}} &= \sup_{X\in  \f_{\mathrm{dp}}(\de) }\norm{\phi(X)}
\notag\\
&\ge \sup_{X\in  \f_{\mathrm{dp}}(\de) \ \text{and} \ X \ \text{ is a scalar operator} \ }\norm{\phi(X)} \notag\\
&= \sup_{\lambda \in R_\delta}{|\phi(\lambda)|} =\|\phi\|_{\hinf (R_\de)}.
\end{align}
\end{rem}

\begin{ex}\label{dp-neq-infty}
Consider the function $f\in\hol (R_\de)$ defined by $f(z)=z +\frac{\delta}{z}$.   
Then 
\[
\| f \|_{\mathrm{dp}} = 2 \ \text{and}\ \| f \|_{\hinf (R_\de)}  = 1+\de.
\]
Moreover, the Crouzeix universal constant $ K_{\mathrm {bfd}} \geq 2$.
\begin{proof}
If $\phi(z) = z$ for $z \in G_\de$  and $\pi: R_\de \to G_\de$ is defined by $\pi(z)=z +\frac{\delta}{z}$, then 
\[\phi \circ \pi (z)=z +\frac{\delta}{z} =f(z) \ \text{ for} \ z \in R_\de.\]
By Theorem \ref{dp.thm.10}, 
\[\|\phi\|_{\mathrm {bfd}} = \|\phi \circ \pi\|_{\mathrm {dp}}.\]
 By  \cite[Example 4.26]{aly24B}, $\| \phi \|_{\mathrm{bfd}} = 2$.
Therefore, 
\[
\| f \|_{\mathrm{dp}} =  \|\phi \circ \pi\|_{\mathrm {dp}}= \| \phi \|_{\mathrm{bfd}} = 2.
\]
Note that 
\[
\| f \|_{\hinf (R_\de)} =\sup_{z \in R_\de} \left| z +\frac{\delta}{z} \right| = 1+\de.
\]
Note that  $\phi(z)=z$ has bfd-norm equal to $2$ and sup norm on $G_\de$ equal to $1+\de$.
Hence the Crouzeix universal constant $ K_{\mathrm {bfd}} \geq 2$. \end{proof}

\end{ex}

\begin{rem} In \cite{Ts2022}  G. Tsikalas proved a result about the annulus as a $K$-spectral set. We restate his result in the notation of this paper as follows.
Let $K(\de)$ denote the smallest constant such that $R_\de$ is a $K(\de)$-spectral set for any bounded linear operator $T \in \fdp(\delta)$.
He used the functions $g_n$ in $\hol(R_\de)$ defined by
$$g_n(z) = \frac{\de^n}{z^n} + z^n, \quad \text{for}\  n=1, 2, \dots,$$
to show that $K(\de) \ge 2$, for all $\de\in(0,1)$.
\end{rem}

\begin{prop} \label{phi-in-holD} 
If $\phi \in \hol(\d)$, then 
\begin{align} \label{dp-sup-D}
\norm{\phi}_{\mathrm {dp}} &= \sup_{X\in  \f_{\mathrm{dp}}(\de) }\norm{\phi(X)}
\notag\\
&=\|\phi\|_{\hinf (R_\de)} =\|\phi\|_{\hinf(\d)}.
\end{align}
\end{prop}

\begin{proof} By the  definition of the dp-norm,
\begin{align} \label{dp-sup-D-1}
\norm{\phi}_{\mathrm {dp}} &= \sup_{X\in  \f_{\mathrm{dp}}(\de) }\norm{\phi(X)}
\notag\\
&\leq \sup_{\|X\| \le 1} \norm{\phi(X)} \  & \text{by the definition of }\f_{\mathrm{dp}}(\de)
\notag\\
&= \|\phi\|_{\hinf(\d)}.\ &  \text{by von Neumann's  inequality}
\end{align}
By the Maximum principle, for $\phi \in \hol(\d)$,
\be \label{maximum-pr}
\|\phi\|_{\hinf (R_\de)}= \|\phi\|_{\hinf(\d)}.
\ee
By inequality  \eqref{dp-sup}, $\norm{\phi}_{\mathrm {dp}} \geq \|\phi\|_{\hinf (R_\de)}$ and, by inequality  \eqref{dp-sup-D-1}, 
 $\norm{\phi}_{\mathrm {dp}} \leq \|\phi\|_{\hinf(\d)}$, and so 
\be
\norm{\phi}_{\mathrm {dp}} =\|\phi\|_{\hinf(\d)}.
\ee
Therefore, the equalities \eqref{dp-sup-D} hold. \end{proof}

\section{Models and realizations of holomorphic functions on $R_\delta$} \label{model-dp}

 In this section we review some known results on the function theory  of holomorphic functions in the dp-norm on an  annulus. The  models and realizations of holomorphic functions $\phi:R_\de\to\c$ such that
$\|\phi\|_{\mathrm {dp}} \leq 1$ are presented in 
\cite[Theorem 9.46]{amy20}. The theorem  states the following.

\begin{thm}\label{dpmodel}
Let $\delta \in (0,1)$. Let $\phi:R_\de\to\c$ be holomorphic and satisfy $\|\phi\|_{\mathrm {dp}} \leq 1$. There exists a $\mathrm {dp}$-model $(\n,v)$ of $\phi$ with parameter $\de$, in the sense that there are Hilbert spaces $\n^+,\n^-$ and an ordered pair $v=(v^+,v^-)$ of holomorphic functions, where 
$v^+:R_\de \to \n^+$ and $v^-:R_\de\to \n^-$ satisfy, for all $z,w \in R_\de$,
\[
1-\overline{\phi(w)}\phi(z) = (1-\overline w z)\ip{v^+(z)}{v^+(w)}_{\n^+} + (\overline w z-\de^2)\ip{v^-(z)}{v^-(w)}_{\n^-}.
\]
\end{thm}

\begin{defin}
A {\em positive semi-definite function} on a set $X$ is a function $A: X\times X \to \c$ such that, for any positive integer $n$ and any points $x_1,\dots,x_n \in X$, the $ n\times n$ matrix $[A(x_j,x_i)]_{i,j=1}^n$ is positive semi-definite.

  We shall write 
\[
  [A(x,y)] \geq 0, \ \text{for all} \ x,y\in X,
\]
to mean that $A$ is a positive semi-definite function on $X$.
\end{defin}

\begin{thm}\label{thm.10}
$\phi \in \essdp$ if and only if there exist a pair of positive semi-definite functions $A$ and $B$ on $R_\delta$ such that 
\be\label{100}
1-\overline{\phi(\mu)}\phi(\lambda)=
(1-\overline\mu \lambda)A(\lambda,\mu)+
(1-\frac{\delta}{\overline{\mu}}\frac{\delta}{\lambda})B(\lambda,\mu)
\ee
for all $\lambda,\mu \in R_\delta$.  
\end{thm}
\begin{proof}
 For a proof see Definition 9.44 and Theorem 9.46 in \cite{amy20}. \end{proof}

Recall Moore's Theorem \cite[Theorem 2.5]{amy20}: if $\Omega$ is a set and $A:\Omega \times \Omega \to \c$ is a function, then $A$ is a positive semi-definite function on $\Omega$ if and only if there exists a Hilbert space $\m$ and a function $u:\Omega \to \m$ satisfying
\be\label{110}
A(\lambda,\mu)=\ip{u(\lambda)}{u(\mu)}_\m
\ee
for all $\lambda,\mu \in \Omega$. 
Thus, if $A$ and $B$ are as in equation \eqref{100}, we may choose Hilbert spaces $\m_1$ and $\m_2$ such that 
\[
A(\lambda,\mu)=\ip{u_1(\lambda)}{u_1(\mu)}_{\m_1}\ \text{ and }\ 
B(\lambda,\mu)=\ip{u_2(\lambda)}{u_2(\mu)}_{\m_2}
\] 
for all $\lambda,\mu \in R_\delta$. If we then let $\m=\m_1 \oplus \m_2$ and define $E:R_\delta \to \b(\m)$ and $u:R_\delta \to \m$ by the formulae
\be\label{formE}
E(\lambda)=\begin{bmatrix}\lambda & 0\\ 0 & \frac{\delta}{\lambda}\end{bmatrix}\ \text{ and }\  
u(\lambda)=\begin{bmatrix}u_1(\lambda)\\ u_2(\lambda)\end{bmatrix},\qquad\text{for}\  \lambda \in R_\delta, 
\ee
then the relation \eqref{100} becomes the formula
\be\label{120}
1-\overline{\phi(\mu)}\phi(\lambda) = \Bigip{\big(1-E(\mu)^*E(\lambda)\big)\ u(\lambda)}{u(\mu)}_\m\ 
\text{for}\  \lambda,\mu \in R_\delta.
\ee

When $A$ is positive semi-definite, let us agree to say that $A$ \emph{has finite rank} if $\m$ in the formula \eqref{110} can be chosen to have finite dimension. In this case, we may define $\rank (A)$ by setting
\[
\rank (A) = \dim \m
\]
where $\m$ satisfying  \eqref{110} is chosen to have minimal dimension.\footnote{Equivalently, $\set{u(\lambda)}{\lambda \in \Omega}$ spans $\m$.}

The following theorem is stated as  \cite[Theorem 9.54]{amy20}. For the convenience of the reader we shall give a full proof here. 

\begin{thm}\label{thm3.8} {\rm \bf A realization formula.} Let $\phi \in   \mathcal{S}_{\mathrm {dp}}(R_\de)$.
If $(\m, u)$ is a model for $\phi$ then there exists a unitary operator $L \in \b(\c \oplus \m)$ such that if we decompose $L$ as a block operator matrix
\be\label{130}
L=\begin{bmatrix}a&1\otimes\beta\\
\gamma\otimes 1& D\end{bmatrix},
\ee
where $a \in \c$, $\beta\in \m$, $\gamma \in \m$, and $D \in \b(\m)$, then
\be\label{140} 
\phi(\lambda)=a + \Bigip{E(\lambda)\big(1-DE(\lambda)\big)^{-1}\ \gamma}{\beta}_\m,\qquad \text{for all}\ \lambda \in R_\delta.
\ee
Conversely, if $a \in \c$, $\beta\in \m$, $\gamma \in \m$, and $D \in \b(\m)$ are such that $L$ as defined by equation \eqref{130} is unitary and if $\ph$ is given by equation \eqref{140} and $u:R_\delta \to \m$ is defined by 
\be\label{150}
u(\lambda)=\big(1-DE(\lambda)\big)^{-1}\gamma,\qquad \text{for}\ \lambda \in R_\delta,
\ee
then $(\m,u)$ is a model for $\phi$. 
\end{thm}
\begin{proof}
Let $(\m,u)$ be a model for $\ph$.  As explained in Theorem \ref{dpmodel}, it means that there exist Hilbert spaces $\n^+$ and $\n^-$ and maps $v^+: R_\de \to \n^+, v^-:R_\de \to \n^-$ such that, for all $\lam,\mu\in R_\de$,
\[
1- \ov{\ph(\mu)}\ph(\lam) = (1-\overline\mu\lam)\ip{v^+(\lam)}{v^+(\mu)}_{\n^+} + (\overline\mu\lam-\de^2)\ip{v^-(\lam)}{v^-(\mu)}_{\n^-}.
\]
Reshuffle this relation to 
\begin{align*}
1+ &\ip{\lam v^+(\lam)}{\mu v^+(\mu)}_{\n^+} + \ip{\de v^-(\lam)}{\de v^-(\mu)}_{\n^-} \\
&= \ov{\ph(\mu)}\ph(\lam) + \ip{v^+ (\lam)}{v^+ (\mu)}_{\n^+} +\ip{\lam v^- (\lam)}{\mu v^-(\mu)}_{\n^-},
\end{align*}
and notice that this equation amounts to saying that the following families of vectors in $\c\oplus \n$, where $\n \df \n^+ \oplus \n^-$,
\[
\bpm 1 \\ \lam v^+(\lam) \\ \de v^-(\lam) \epm_{\lam \in R_\de} \quad \text{and}\quad \bpm \ph(\lam)\\ v^+(\lam) \\ \lam v^-(\lam) \epm_{\lam\in R_\de}
\]
 have the same gramian.  Let the closed linear spans of these two families be $\x$ and $\y$ respectively.
 By the Lurking Isometry Lemma \cite[Lemma 2.18]{amy20} there exists a linear isometry $L:\x \to \y$ such that
 \be\label{Lprop}
 L\bpm 1 \\ \lam v^+(\lam) \\ \de v^-(\lam) \epm = \bpm \ph(\lam)\\ v^+(\lam) \\ \lam v^-(\lam) \epm
 \ee
 for all $\lam\in R_\de$.  Since both $\x$ and $\y$ are subspaces of $\c\oplus\n$, we may extend $L$ (possibly after enlarging the space $\n$) to a unitary operator  $L:\n \to \n$ (see the discussion in \cite[Remark 2.31]{amy20} for this step).  Write $L$ as a block operator matrix
 \be\label{Lform}
 L \sim \bpm a & 1\otimes \beta \\ \ga\otimes 1& D \epm
 \ee
 with respect to the orthogonal decomposition $\c\oplus(\n^+\oplus\n^-)$ of $\c\oplus\n$ and define a map $u:R_\de \to \n$ by
\[
u(\lam) = \bpm v^+(\lam) \\ \lam v^-(\lam) \epm.
\]
   Then equation \eqref{Lprop} yields the relations
 \begin{align}\label{3.141}
 a + \ip{E(\lam) u(\lam)}{\beta}_{\n} &= \ph(\lam) \\
 \ga + DE(\lam)u(\lam) &= u(\lam),
 \end{align}
 where $E(\lam)$ is given by equation \eqref{formE}.  Since $\|D\|\leq 1$ and 
 $$\|E(\lam)\|= \max{\left\{ |\lam|, \frac{\de}{|\lam|} \right\} } < 1  \ \text{ for all} \  \lam \in R_\de,$$ 
 it follows that $1-DE(\lam)$ is invertible for $\lam\in R_\de$, 
and hence
 \begin{align*}
 u(\lam) &= (1-DE(\lam))\inv\ga, \\
 \ph(\lam)&= a + \ip{E(\lam)(1-DE(\lam))\inv\ga}{\beta}
 \end{align*}
 for all $\lam\in R_\de$, which is the desired realization formula \eqref{140}.
 
 Conversely, suppose that $a,\beta,\gamma, D$  are such that $L$ given by equation \eqref{Lform} is  a unitary operator  on $\c\oplus \n$ and that $\ph$ is the function on $R_\de$ defined by equation \eqref{140}.
 Since $1-DE(\lam)$ is invertible for all $\lam\in R_\de$ we may define a mapping $u:R_\de \to\n$ by equation \eqref{150}.  Then the equations \eqref{3.141} hold.  They may be written in the form
 \[
 L\bbm 1\\E(\lam) u(\lam) \otimes 1 \ebm = \bbm \ph(\lam)\otimes 1 \\ u(\lam)\otimes 1 \ebm \ \text{for}\  \lam\in R_\de.
 \]
 Thus, for any $\mu \in R_\de$,
 \[
  \bbm 1 & 1\otimes E(\mu)u(\mu) \ebm L^* = \bbm 1\otimes \ph(\mu) & 1\otimes u(\mu) \ebm.
 \]
 Multiply the last two displayed equations together and use the fact that $L^*L=1$ to infer that, for any $\lam,\mu\in R_\de$,
 \[
 \bbm 1 & 1\otimes E(\mu)u(\mu) \ebm \bbm 1\\E(\lam) u(\lam) \otimes 1 \ebm = \bbm 1\otimes \ph(\mu) & 1\otimes u(\mu) \ebm \bbm \ph(\lam)\otimes 1 \\ u(\lam)\otimes 1 \ebm,
 \]
 which multiplies out to give the relation, for all $\lam,\mu\in R_\de$,
 \[
 1-\ov{\ph(\mu)}\ph(\lam) = \ip{(1-E(\mu)^*E(\lam))u(\lam)}{u(\mu)}_{\n},
 \]
 that is, $(\n,u)$ is a DP-model of $\ph$. \end{proof}

Let us recall the interpolation problem we posed in the Introduction.

\begin{defin}\label{DPPick}  {\bf The DP Pick Problem.} 
Given $n$ distinct points $\lambda_1,\ldots,\lambda_n$ in $R_\delta$ and $z_1,\ldots,z_n \in \c$, does there exist a function $\phi \in \hinf_{\mathrm {dp}}(R_\de)$ 
with $\norm{\phi}_{\mathrm {dp}} \le 1$ such that 
\be\label{200}
\phi(\lambda_j) = z_j \qquad \text{for} \ j=1,\ldots,n ?
\ee
 
We say the DP Pick Problem \eqref{DPPick} is \emph{solvable} if there exists $\phi \in \essdp$ satisfying equations \eqref{200}.
\end{defin}

The following theorem, which is a Pick interpolation theorem in the dp norm, is \cite[Theorem 9.55]{amy20}.

\begin{thm}\label{Pick-dp} 
Let $\delta \in (0,1)$. Let $\lambda_1, \dots, \lambda_n$ be distinct points in $R_\de$ and let $z_1, \dots, z_n$ be arbitrary complex numbers. There exists $f \in \hinf_{\mathrm {dp}}(R_\de)$ such that
$\|f\|_{\mathrm {dp}} \leq 1$  and
\[ f(\lambda_i) =z_i \;\; \text{for} \;\; i=1, \dots, n, \]
if and only if there exist a pair of $n \times n$ positive semi-definite matrices 
$A = [ a_{ij}]$ and $B= [ b_{ij}]$
such that 
\[
1-\overline{z_i} z_j =(1-\overline \lambda_i \lambda_j)a_{ij}
 + (1 -\frac{\de^2}{\overline \lambda_i \lambda_j})b_{ij}
\]
for $i,j=1, \dots, n$.

\end{thm}

We also assert a dual theorem in terms of ``DP Szeg\H{o} kernels", which we discuss in the next section.

\section{DP-Szeg\H{o} kernels and normalized DP-Szeg\H{o} kernels for the tuple $(\lam_1,\dots,\lam_n)$}\label{dpk}

In this section we follow Abrahamse's idea of using families of kernels to solve Pick interpolation problems. 
To this end we shall introduce several objects that depend on an $n$-tuple $\lam=(\lam_1,\dots,\lam_n)$ of distinct points in $R_\de$.  First we consider the set $\fdp(\delta,\lambda)$ of operators   on $n$-dimensional Hilbert space  with spectrum $\{\lam_1,\dots,\lam_n\}$ which belong to the Douglas-Paulsen family $\fdp(\delta)$. Secondly we define DP Szeg\H{o} kernels for the $n$-tuple $\lambda$.
We establish a close connection between these two objects in Propositions \ref{T-g-G(lam)} and \ref{g-to-T-DP}. Thereby, in Section \ref{dp-Pick} we shall  establish a theorem analogous to Theorem  \ref{Abrahamse}, Abrahamse's Theorem.

\begin{defin} \label{dpkernel}
We say that a kernel $k:R_\delta \times R_\delta$ is a \emph{DP Szeg\H{o} kernel on $R_\de$} if
\be\label{300}
[(1-\overline\mu \lambda)k(\lambda,\mu)]\ge 0\ \ \text{ and }\ \ 
[(1-\frac{\delta}{\overline{\mu}} \frac{\delta}{\lambda})k(\lambda,\mu)] \ge 0,\qquad \text{for all} \; \lambda,\mu \in R_\delta.
\ee
We let
\[
\calk =\set{k}{k \text{ is a DP Szeg\H{o} kernel on $R_\de$}}.
\]
\end{defin}

\begin{defin}\label{T-D-P-lam}
Let $\lambda=(\lambda_1,\ldots,\lambda_n)$ be an $n$-tuple of distinct points in $R_\delta$. We denote by $\fdp(\delta,\lambda)$ the family of operators $T$ in the Douglas-Paulsen family  $\f_{\mathrm{dp}}(\de)$ corresponding to the annulus $R_\de$
 that act on an $n$-dimensional Hilbert space  $\h_T$ and satisfy 
\[
\sigma(T)=\{\lambda_1,\ldots,\lambda_n\}.
\]
\end{defin}

If $T\in \fdp(\delta,\lambda)$, then, as $\dim \h_T=n$ and $\sigma(T)$ consists of $n$ distinct points, $T$ is diagonalizable, that is, there exist $n$ linearly independent vectors $e_1,\ldots,e_n \in \h_T$ such that
\be \label{T-e}
Te_j=\lambda_j e_j \qquad \text{for} \ j=1,\ldots,n.
\ee
Let $g$ denote the gramian of the vectors $e_1,\ldots,e_n$, that is,
\be \label{g-of-T}
g=[g_{ij}], \text{ where } g_{ij}=\ip{e_j}{e_i}\qquad \text{for} \ i,j=1,\ldots,n, 
\ee
Then, we shall prove in Proposition \ref{T-g-G(lam)} that  $g=[g_{ij}]$ is a positive definite $n \times n$ matrix such that 
\be\label{210-0}
[(1-\overline\lambda_i \lambda_j)g_{ij}] \ge 0\ \ \text{ and }\ \ 
\left[\left(1-\frac{\delta}{\overline{\lambda_i}} \frac{\delta}{\lambda_j}\right)g_{ij}\right] \ge 0.
\ee

\begin{defin}\label{dp-kernels}
Let $\lambda_1, \dots, \lambda_n$ be $n$ distinct points in $R_\de$.
We define $\mathcal{G}_{\mathrm {dp}} (\lambda)$ to be the set of positive definite $n \times n$ matrices $g = [g_{ij}]$ such that 
\be\label{210}
[(1-\overline\lambda_i \lambda_j)g_{ij}] \ge 0\ \ \text{ and }\ \ 
\left[\left(1-\frac{\delta}{\overline{\lambda_i}} \frac{\delta}{\lambda_j}\right)g_{ij}\right] \ge 0.
\ee
We call $g \in \mathcal{G}_{\mathrm {dp}} (\lambda)$ a  {\em DP-Szeg\H{o} kernel} for the $n$-tuple $\lam=(\lam_1,\dots,\lam_n)$. \\
\end{defin}

\begin{prop} \label{T-g-G(lam)} Let $\lambda_1, \dots, \lambda_n$ be $n$ distinct points in $R_\de$.
Let  $T\in \fdp(\delta,\lambda)$. Then the  gramian $g=[g_{ij}]$ of vectors $e_1,\ldots,e_n$ that satisfy the equations \eqref{T-e} and \eqref{g-of-T}
is a positive definite $n \times n$ matrix which belongs to $\mathcal{G}_{\mathrm {dp}} (\lambda)$.
\end{prop}

\begin{proof} By assumption $T$ is a  Douglas-Paulsen operator  with parameter $\de$ that acts on an $n$-dimensional Hilbert space  $\h_T$, $T$ has $n$ linearly independent eigenvectors $e_1,\dots,e_n$ corresponding to the eigenvalues $\lam_1,\dots,\lam_n$ respectively and $g_{ij} = \ip{e_j}{e_i}$ for $i,j=1,\dots,n$.
By the definition of the Douglas-Paulsen class, $\norm{T} \le 1$ and $\norm{\delta T\inv} \le 1$, so that, for any vector $x=\sum_{j=1}^n x_j e_j$, we have
\begin{align*}
0 & \leq \|x\|^2 - \|Tx\|^2 \\
&= \ip{\sum_{j=1}^n x_j e_j}{\sum_{i=1}^n x_i e_i} -  \ip{\sum_{j=1}^n x_j\lam_j e_j}{\sum_{i=1}^n \lam_i x_i e_i}  \\
&= \sum_{i,j=1}^n\ov{x_i}\left((1-\ov{\lam_i}\lam_j)g_{ij}\right)x_j \\
&= \bbm \ov{x_1} & \dots & \ov{x_n} \ebm \bbm (1-\ov{\lam_i}\lam_j)g_{ij} \ebm \bbm x_1 \\\dots \\x_n \ebm.
\end{align*}
Thus
\[
\big[(1-\overline\lambda_i\lambda_j)g_{ij}\big]_{i,j=1}^n \ge 0. \ 
\]
Likewise, the relation $\|\delta T\inv x\| \le \|x\|$ holds for any vector $x=\sum_{j=1}^n x_j e_j \in  \h_T$. Therefore, we have
\begin{align*}
0 & \leq \|x\|^2 - \left\|\delta T\inv x\right\|^2 \\
&= \ip{\sum_{j=1}^n x_j e_j}{\sum_{i=1}^n x_i e_i} -  \ip{ \sum_{j=1}^n \frac{\de}{\lam_j} x_j e_j }{ \sum_{i=1}^n  \frac{\de}{\lam_i} x_i e_i }  \\
&= \sum_{i,j=1}^n\ov{x_i}\left(\left(1-\frac{\de}{\ov{\lam_i}}\frac{\de}{\lam_j}\right)g_{ij}\right)x_j \\
&= \bbm \ov{x_1} & \dots & \ov{x_n} \ebm \bbm (1-{\frac{\de}{\ov{\lambda_i}}} \frac{\de}{\lambda_j})
g_{ij} \ebm \bbm x_1 \\\dots \\x_n \ebm.
\end{align*}
Thus
\[
\big[\left(1-{\frac{\de}{\ov{\lambda_i}}} \frac{\de}{\lambda_j}\right) g_{ij}\big]_{i,j=1}^n \ge 0.
\]
Therefore, $g=[g_{ij}]$ is a positive definite DP-Szeg\H{o} kernel for the $n$-tuple $\lam=(\lam_1,\dots,\lam_n)$. \end{proof}

Let $g \in \mathcal{G}_{\mathrm {dp}} (\lambda)$, so that $g > 0$. By Moore's theorem \cite[Theorem 2.5]{amy20}, $g$ is the gramian  matrix of a basis $e_1, \dots, e_n$ of an $n$-dimensional Hilbert space ${\mathcal H}$.

\begin{prop} \label{g-to-T-DP} Let $\lambda_1, \dots, \lambda_n$ be $n$ distinct points in $R_\de$.
Let $g \in \mathcal{G}_{\mathrm {dp}} (\lambda)$. Let $g$ be the gramian  matrix of a basis $e_1, \dots, e_n$ of an $n$-dimensional Hilbert space ${\mathcal H}$. Define $T \in \b(\h)$ by 
\be \label{defT}
T e_j= \lambda_j e_j, \ j =1, \dots, n.
\ee
Then 
$T \in \f_{\mathrm{dp}}(\delta,\lambda)$.
\end{prop}
\begin{proof}
Let us show that $T$ is a Douglas-Paulsen operator. 
If 
$ x= \sum_{j=1}^n \xi_j e_j \in \h $, $ Tx = \sum_{j=1}^n \xi_j \lambda_j e_j$ and
\begin{align}\label{T-lambda}
\|Tx\|^2 &= \big\langle\, \sum_{j=1}^n \xi_j \lambda_j e_j ,
\sum_{i=1}^n \xi_i \lambda_i e_i\, \big\rangle \notag\\
&=\sum_{j,i=1}^n \xi_j \lambda_j \overline\xi_i \overline\lambda_i \langle\, e_j,e_i \,\rangle \notag\\
&=\sum_{j,i=1}^n \overline\lambda_i \lambda_j  \xi_j \overline\xi_i  \, g_{ij}.
\end{align}
Hence
\[
\|x\|^2 - \|Tx\|^2 =\sum_{j,i=1}^n (1 - \overline\lambda_i \lambda_j)  \, g_{ij} \xi_j  \overline\xi_i .
\]
By hypothesis, 
\[ [(1-\overline\lambda_i \lambda_j)g_{ij}] \ge 0, \]
and so $\|x\|^2 - \|Tx\|^2 \ge 0$. Thus $\|T\|  \le 1$.
Similarly, using the hypothesis 
\[
\left[\left(1-\frac{\delta}{\overline{\lambda_i}} \frac{\delta}{\lambda_j}\right)g_{ij}\right] \ge 0,
\]
one can show that $\|\delta T^{-1}\|  \le 1$. Therefore 
$T$ is a Douglas-Paulsen operator.
In addition, by the definition \eqref{defT} of $T$ ,
\[
\sigma(T)=\{\lambda_1,\ldots,\lambda_n\} \subset R_\de.
\]
thus $T \in \f_{\mathrm{dp}}(\delta,\lambda)$. \end{proof}

\begin{prop} \label{dp-kernel-z} Let  $\lambda_1,\ldots,\lambda_n$ be $n$ distinct points in $R_\delta$ and $z_1,\ldots,z_n \in \c$.
If the  DP Pick Problem \ref{DPPick} is solvable, then, for any positive definite
$g \in \mathcal{G}_{\mathrm {dp}} (\lambda)$,
\be\label{211}
[(1-\overline z_i z_j)g_{ij}] \ge 0.
\ee
\end{prop}
\begin{proof}
By assumption, the  DP Pick Problem \ref{DPPick} is solvable, that is, 
there exist a function $\phi \in \hinf_{\mathrm {dp}}(R_\de)$ 
with $\norm{\phi}_{\mathrm {dp}} \le 1$ and satisfying 
\be\label{220}
\phi(\lambda_j) = z_j,\qquad j=1,\ldots,n.
\ee
Let $g \in \mathcal{G}_{\mathrm {dp}} (\lambda)$, and so $g > 0$. By Moore's theorem \cite[Theorem 2.5]{amy20}, $g$ is the gramian  matrix of a basis $e_1, \dots, e_n$ of an $n$-dimensional Hilbert space ${\mathcal H}$. Define $T \in \b(\h)$ by 
\be \label{defT1}
T e_j= \lambda_j e_j, \ j =1, \dots, n.
\ee
By Proposition \ref{g-to-T-DP}, $T \in \f_{\mathrm{dp}}(\delta,\lambda)$.
By assumption,
$\phi \in \essdp$, and so $\|\phi(T)\| \le 1$. 
For any 
$ x= \sum_{j=1}^n \xi_j e_j \in \h $,
\begin{align}\label{phiT-lambda}
\phi(T) x & = \phi(T)\sum_{j=1}^n \xi_j e_j \notag\\
&=\sum_{j=1}^n \xi_j \phi(\lambda_j) e_j = \sum_{j=1}^n \xi_j z_j e_j.
\end{align}
Therefore, by  equation \eqref{phiT-lambda}, the condition $\|\phi(T)\| \le 1$
translates into 
\[
[(1-\overline z_i z_j)g_{ij}] \ge 0. 
\qedhere
\]
\end{proof}

\begin{defin}\label{norm-dp-kernels}
Let $\lambda_1, \dots, \lambda_n$ be $n$ distinct points in $R_\de$.
We say that a DP-Szeg\H{o} kernel $[g_{ij}] \in \mathcal{G}_{\mathrm {dp}} (\lambda)$ is {\em normalized} if $g_{ii}=1$ for $i=1,\dots,n$. \\ 
Let $\mathcal{G}_{\mathrm {dp}}^{\mathrm{norm}} (\lambda)$ denote the set of normalized DP-Szeg\H{o} kernels for the $n$-tuple $(\lam_1,\dots,\lam_n)$.
\end{defin}

\begin{rem}\label{norm-kernel} Let $\lambda_1, \dots, \lambda_n$ be $n$ distinct points in $R_\de$.
Every DP-Szeg\H{o} kernel $[g_{ij}]$ from $\mathcal{G}_{\mathrm {dp}} (\lambda)$ is diagonally congruent to a normalized DP-Szeg\H{o} kernel.
\end{rem}
\begin{proof} For any matrix $[g_{ij}]\in \mathcal{G}_{\mathrm {dp}} (\lambda)$, we can define a positive definite matrix $[h_{ij}]$ by 
\be\label{defh}
h_{ii}=1\ \text{for}\ i=1,\dots,n \ \text{and}\ h_{ij}= c_i\inv g_{ij} c_j\inv\ \text{if}\ i\neq j
\ee
 where 
\be\label{defci}
c_i = \sqrt{g_{ii}}\  \text{if}\  g_{ii}\neq 0\ \text{and} \ c_i=1\  \text{if}\ g_{ii}=0.
\ee  
Then $h_{ii}=1$ for each $i$, and 
\be \label{defC}
\bbm h_{ij} \ebm_{i,j=1}^n = C^*\bbm g_{ij} \ebm_{i,j=1}^n C \ \text{where}\ C= \mathrm{diag} \{1/c_1, \dots, 1/c_n\}.
\ee
On conjugating the inequalities \eqref{210} by the matrix $C$ we find that $[h_{ij}]$ belongs to $\mathcal{G}_{\mathrm {dp}}^{\mathrm{norm}} (\lambda)$. \end{proof}

\begin{prop} \label{dpcompact} Let $\lambda_1, \dots, \lambda_n$ be $n$ distinct points in $R_\de$.
The set $\mathcal{G}_{\mathrm {dp}}^{\mathrm{norm}} (\lambda)$ is compact in the topology of the space of $n\times n$ complex matrices. Moreover, for fixed target data $z_1,\dots,z_n$,
\be \label{more}
\bbm (1-\overline z_i z_j) g_{ij} \ebm_{i,j=1}^n \geq 0\  \text{for all}\ g\in \mathcal{G}_{\mathrm {dp}} (\lambda)
\ee
if and only if
\be \label{less}
\bbm (1-\overline z_i z_j) g_{ij} \ebm_{i,j=1}^n \geq 0\  \text{for all}\ g\in \mathcal{G}_{\mathrm {dp}}^{\mathrm{norm}} (\lambda).
\ee
\end{prop}
\begin{proof}
Consider any matrix $g=[g_{ij}] \in \mathcal{G}_{\mathrm {dp}}^{\mathrm{norm}} (\lambda)$.  Since $g$ is positive definite, the principal minor on rows $i$ and $j$ is non-negative, which is to say that $1-|g_{ij}|^2 \geq 0$ for $i,j=1,\dots,n$.  It follows that the operator norm $\|g\| \leq n$, and so $\mathcal{G}_{\mathrm {dp}}^{\mathrm{norm}} (\lambda)$ is bounded.  Let us prove that $\mathcal{G}_{\mathrm {dp}}^{\mathrm{norm}} (\lambda)$ is sequentially compact.

Let $g^\ell$, $\ell =1,2, \dots$, be a sequence in $\mathcal{G}_{\mathrm {dp}}^{\mathrm{norm}} (\lambda)$.
We claim that $(g^\ell)_{\ell \ge 1}$ has a subsequence that converges to an element of $\mathcal{G}_{\mathrm {dp}}^{\mathrm{norm}} (\lambda)$. For each $\ell$, since $g^\ell$ is non-singular, by the definition of $\mathcal{G}_{\mathrm {dp}}^{\mathrm{norm}} (\lambda)$, we may pick a basis $e_1^\ell, \dots, e_n^\ell$ of $\c^n$ such that $g^\ell$ is the gramian  matrix of the basis $e_1^\ell, \dots, e_n^\ell$, which is to say that
\be \label{g-of-T-comp}
g^\ell=[g_{ij}^\ell], \text{ where } g_{ij}^\ell=\ip{e_j^\ell}{e_i^\ell}\qquad \text{for} \ i,j=1,\ldots,n. 
\ee
Define $T^\ell \in \b(\c^n)$ by 
\be \label{defTcomp}
T^\ell  e_j^\ell = \lambda_j e_j^\ell , \ j =1, \dots, n.
\ee
Note that since $g^\ell$  is normalised, that is,
\[
g_{ii}^\ell=\ip{e_i^\ell}{e_i^\ell}  = 1 \qquad \text{for} \ i=1,\ldots,n,
\]
and so
$\|e_i^\ell \| = 1 $ for $i=1,\ldots,n$. Then, by Proposition \ref{g-to-T-DP}, 
\[
\sigma(T^\ell)=\{\lambda_1,\ldots,\lambda_n\}
\]
and $T^\ell \in \f_{\mathrm{dp}}(\delta,\lambda)$.
By the compactness of the unit sphere in $\c^n$,
we can choose a subsequence $(e_i^{\ell_k})_{k \ge 1} $ of $(e_i^{\ell})_{\ell \ge 1} $ such that 
 $(e_j^{\ell_k})$ converges to a unit vector $v_j \in \c^n$ as  $k \to \infty$ for $j=1,\dots,n$.
  By the compactness of  the unit ball in $\b(\c^n)$, by passing to a further subsequence $(e^{\ell_k})_{k\geq 1}$ of $(e^\ell)_{\ell\geq 1}$ we can arrange also that
 $(T^{\ell_k})$ converges to a limit $T \in \b(\c^n)$ as  $k \to \infty$.
In the relations 
\be \label{defTcomp-2}
T^{\ell_k}  e_j^{\ell_k} = \lambda_j e_j^{\ell_k}, \ j =1, \dots, n,
\ee
let $k \to \infty$ to obtain 
\be \label{defTcomp-3}
T  v_j = \lambda_j v_j \qquad \text{and } \ \|v_j \| = 1\ \qquad j =1, \dots, n.
\ee
Thus 
\[
\sigma(T^\ell)=\{\lambda_1,\ldots,\lambda_n\},
\]
the eigenvectors $v_1,\dots,v_n$ of $T$ corresponding to the distinct eigenvalues $\lambda_1,\ldots,\lambda_n$ are linearly independent and therefore span $\c^n$,
and $T^\ell \in \f_{\mathrm{dp}}(\delta,\lambda)$.
Let $g$ be the Gramian of the vectors $v_1, \dots, v_n$ in $\c^n$: then $g$ is positive definite, and
by Proposition \ref{T-g-G(lam)}, $g\in \mathcal{G}_{\mathrm {dp}}^{\mathrm{norm}} (\lambda)$.
We have
\[
g_{ij}=\ip{v_j}{v_i}= \lim_{k \to\infty} \ip{v_j^{\ell_k}}{v_i^{\ell_k}} = \lim_{k\to\infty} g_{ij}^{\ell_k}
\]
for $i,j=1,\dots,n$, and so $g^{\ell_k} \to g$ as $k\to\infty$.
We have shown that $\mathcal{G}_{\mathrm {dp}}^{\mathrm{norm}} (\lambda)$ is sequentially compact in 
the metrizable topology of $\b(\c^n)$, hence it is compact.

To prove the ``Moreover", fix target data $z_1,\dots,z_n$.  Since $\mathcal{G}_{\mathrm {dp}} (\lambda) \supset \mathcal{G}_{\mathrm {dp}}^{\mathrm{norm}} (\lambda)$, trivially statement \eqref{more} implies statement \eqref{less}.  Conversely, suppose statement \eqref{less} holds and consider any kernel $g \in \mathcal{G}_{\mathrm {dp}} (\lambda)$.  Define matrices $h=[h_{ij}]$ and $C$ by the relations \eqref{defh}, \eqref{defci} and \eqref{defC}.  Then $h \in \mathcal{G}_{\mathrm {dp}}^{\mathrm{norm}} (\lambda)$, 
and so, by assumption,
\be\label{219}
[(1-\overline z_i z_j)h_{ij}] \ge 0
\ee
Conjugate this matrix inequality by $\mathrm{diag}\{c_1,\dots,c_n\}$ to obtain the relation \eqref{more}.
Thus the relation \eqref{less} implies the relation \eqref{more}. \end{proof}

Say that a DP-Szeg\H{o} kernel $g$ on $R_\de$ is {\em reducible} if there exist  DP-Szeg\H{o} kernels
$h$ and $k$ on $R_\de$ such that $g=h+k$ and neither $h$ nor $k$ is diagonally congruent to $g$.
Here two kernels $g$ and $h$ on $R_\de$ are said to be diagonally congruent if there exists a function $c:R_\de \to \c\setminus\{0\}$ such that, for all $\lam,\mu \in R_\de$, $h(\lam,\mu)= c(\lam)g(\lam,\mu) \ov{c(\mu)}$.  A DP-Szeg\H{o} kernel is {\em irreducible} if it is not reducible.  Clearly, if DP Pick data $\lam_j \mapsto z_j, j=1,\dots,n$, are such that
\[
[(1-\overline z_iz_j)g_{ij}] \geq 0\ \text{and}\  [1-\frac{\de^2}{\ov{z_i}z_j} g_{ij}] \geq 0
\]
for all irreducible DP Szeg\H{o} kernels $g$ then the same inequality holds for {\em all} DP Szeg\H{o} kernels, and consequently the DP pick interpolation problem is solvable. 
Since the class of irreducible DP Szeg\H{o} kernels is likely to be {\em much} smaller than the class of all DP Szeg\H{o} kernels, it would be valuable to identify the irreducible DP Szeg\H{o} kernels on $R_\de$.
\begin{prob}
 Find an effective description of the irreducible DP Szeg\H{o} kernels on $R_\de$.
\end{prob}

\section{The DP Pick problem and DP-Szeg\H{o} kernels}\label{dp-Pick}
In this section we shall prove our main theorem, which is a solvability criterion for DP Pick problems in terms of DP-Szeg\H{o} kernels.  We also present some examples which illustrate the relationship between the Pick and DP Pick interpolation problems.

The following notation and terminology will be needed in the proofs.

\begin{defin}
Let $H_n$ be the real linear space of Hermitian matrices in $\c^{n \times n}$.
A subset $P$ of $H_n$ is called a {\em cone} if the following conditions are satisfied:
(i)  $P +P \subseteq P$, (ii) $ P \cap (-P) = \{0\}$ and (iii)  $\alpha P \subseteq P$ whenever
 $\alpha  \in \r $ and $\alpha \ge 0$.
\end{defin}

\begin{thm} \label{iff-dp-kernel-z}
Let  $\lam_1,\dots,\lam_n \in R_\de$ be distinct  and let $z_1,\dots,z_n\in\c$.
There exists $ \phi \in \essdp$ such that
\[\phi(\lambda_j) = z_j \qquad \text{for} \ j=1,\ldots,n, \]
if and only if, for all 
$g \in \mathcal{G}_{\mathrm {dp}} (\lambda)$,
\be\label{212-1}
[(1-\overline z_i z_j)g_{ij}] \ge 0.
\ee
\end{thm}

\begin{proof}
Implication $\Rightarrow$ follows from Proposition \ref{dp-kernel-z}. 

To prove $\Leftarrow$, suppose that  
\be\label{221}
[(1-\overline z_i z_j)g_{ij}] \ge 0
\ee
for all 
$g \in \mathcal{G}_{\mathrm {dp}} (\lambda)$.

By Theorem \ref{Pick-dp}, to show that  the  DP Pick Problem \eqref{DPPick} is solvable it suffices to prove that  there exist a pair of $n \times n$ positive semi-definite matrices 
$A = [ a_{ij}]$ and $B= [ b_{ij}]$
such that 
\[
1-\overline{z_i} z_j =(1-\overline{\lambda_i} \lambda_j)a_{ij}
 + (\overline \lambda_i \lambda_j-\de^2)b_{ij}
\]
for all $i,j=1, \dots, n$.
Let $H_n$ be the real linear space of Hermitian matrices in $\c^{n \times n}$, and let 
\be\label{Cone_C}
{\mathcal C} = 
\left\{ [(1-\overline{\lambda_i} \lambda_j) a_{ij} ]_{i,j=1}^n +
 \left[\left(1-\frac{\de^2}{\overline{\lambda_i} \lambda_j} \right) b_{ij} \right]_{i,j=1}^n:
[a_{ij}]_{i,j=1}^n \ge 0 \ \text{and} \ [b_{ij}]_{i,j=1}^n \ge 0 \right\}.
\ee
The subset ${\mathcal C}$ is a closed convex cone in $H_n$. 

Note that every $n \times n$ positive semi-definite matrix  $[a_{ij}]_{i,j=1}^n$ belongs to ${\mathcal C}$. By the positivity of Szeg\H{o} kernel $\left[ \frac{1}{1-  \overline{\lambda_i} \lambda_j} \right]_{i,j=1}^n$, the $n \times n$ matrix of the form
\[
\left[ \frac{a_{ij}}{1- \overline{\lambda_i} \lambda_j} \right]_{i,j=1}^n
\]
is also positive semi-definite. In the definition  of ${\mathcal C}$ \eqref{Cone_C}
we may replace $[a_{ij}]_{i,j=1}^n$
by $\left[ \frac{a_{ij}}{1- \overline\lambda_i \lambda_j} \right]_{i,j=1}^n$ and 
$[b_{ij}]_{i,j=1}^n$ by the zero matrix, to deduce that $[a_{ij}]_{i,j=1}^n$ belongs to ${\mathcal C}$.

By the Hahn-Banach theorem, to show that
$[1-\overline{z_i} z_j]_{i,j=1}^n$ belongs to ${\mathcal C}$ it suffices to prove that, for every real linear functional ${\mathcal L}$ on  $H_n$, ${\mathcal L} \ge 0 $ on ${\mathcal C}$ implies ${\mathcal L} ( [1-\overline{z_i} z_j]_{i,j=1}^n ) \ge 0 $.

Extend ${\mathcal L} $ to a complex linear  functional $\widetilde{\mathcal L}$ on $\c^{n \times n}$ by
\[
\widetilde{\mathcal L}(X + {\mathrm i} \, Y) = {\mathcal L}(X) + {\mathrm i} \, {\mathcal L}(Y)
\]
for $X, Y \in H_n$. Now define a pre-inner product $\langle \cdot, \cdot \rangle_L $ on $\c^n$ by 
\[
\langle\, c, d\,\rangle_L = \widetilde{\mathcal L} (c \otimes d)
\]
for $c, d \in \c^n$. Here $c \otimes d \in \c^{n \times n}$, defined by 
\[
(c \otimes d)(x) = \langle \, x, d \,\rangle_{\c^n} c \; \; \; \text{for all} \ x \in \c^n.
\]
Note that, for any $c \in \c^n$, 
\[
\langle\, c, c\,\rangle_L = \widetilde{\mathcal L} (c \otimes c) = {\mathcal L} (c \otimes c) \geq 0.
\]
Let 
\[
{\mathcal N}= \{ x \in \c^n: \, \langle\, x, x\,\rangle_L = 0\}.
\]
Then ${\mathcal N}$ is a subspace of $\c^n$, and $\langle \cdot, \cdot \rangle_L $ induces an inner product on $\c^n /{\mathcal N}$.

Let  $e_1, \dots, e_n$ be the standard basis of $\c^n$ and let  $T \in \b(\c^n)$ defined by 
\be \label{defT-2}
T e_j= \lambda_j e_j, \ j =1, \dots, n.
\ee
Let us construct an operator $\widetilde T$ on $\c^n/{\mathcal N}$ such that $\|\widetilde T\| \leq 1$ and $\|\de \widetilde T\inv\| \leq 1$.
For $ x= \sum_{j=1}^n \xi_j e_j \in \c^n $, we have
\begin{align}\label{T-lambda-2}
\langle\, x, x\,\rangle_L  - \langle\, Tx, Tx\,\rangle_L & = \widetilde{\mathcal L} (x \otimes x) - \widetilde{\mathcal L} (Tx \otimes Tx) \notag\\
& =\widetilde{\mathcal L}\left(  \sum_{j=1}^n \xi_j e_j\,\otimes \sum_{i=1}^n \xi_i e_i \right) -
\widetilde{\mathcal L}\left(  \sum_{j=1}^n \xi_j \lambda_j e_j\,\otimes \sum_{i=1}^n \xi_i \lambda_i e_i \right) \notag\\
& =\widetilde{\mathcal L}\left( \sum_{j,i=1}^n  \xi_j \overline \xi_i e_j\,\otimes   e_i \right) -
\widetilde{\mathcal L}\left( \sum_{j,i=1}^n  \xi_j \lambda_j \overline \xi_i \overline\lambda_i e_j\,\otimes   e_i \right) \notag \\
& =\widetilde{\mathcal L}\left( \sum_{j,i=1}^n (1- \overline\lambda_i \lambda_j )  \overline \xi_i \xi_j  e_j\,\otimes   e_i \right) \notag\\
& =\widetilde{\mathcal L} \left[(1- \overline\lambda_i \lambda_j )  \overline \xi_i \xi_j \right]_{j,i=1}^n \notag\\
& ={\mathcal L} \left[(1- \overline\lambda_i \lambda_j )  \overline \xi_i \xi_j \right]_{j,i=1}^n 
\ge 0 \;\;\; \text {since}\;\;\mathcal{ L} \geq 0 \ \text{on} \  {\mathcal C}.
\end{align}
Thus 
\be\label{5.7a}
\langle\, x, x\,\rangle_L  - \langle\, Tx, Tx\,\rangle_L \geq 0,
\ee
and so $ x \in {\mathcal N}$ implies that $ Tx \in {\mathcal N}$.
Hence $T$ induces an operator $\widetilde T $ on  $\c^n /{\mathcal N}$ by 
\[
\widetilde{T} (x+{\mathcal N}) = T x + {\mathcal N},
\]
and $\|\widetilde{T} (x+{\mathcal N}) \|^2 \leq \|x+{\mathcal N} \|^2$ for all 
$ (x+{\mathcal N}) \in \c^n /{\mathcal N}$, which implies that 
\be\label{1star}
\|\widetilde T\|\leq 1.
\ee

Notice that $\sigma(T)=\{\lam_1,\dots,\lam_n\} \subset R_\de$ and so $T$ is invertible.
Moreover 
\[
\de T\inv e_j = \frac{\de}{\lam_j} e_j \ \text{for} \ j=1,\dots,n,
\]
and so, in the chain of equations leading to equation \eqref{T-lambda-2}, we may replace $T$ by $\de T\inv$ and $\lam_j$ by $\frac{\de}{\lam_j}$ to deduce that, for $x=\sum_{j=1}^n \xi_j e_j \in \c^n$,
\[
\ip{x}{x}_L-\ip{\de T\inv x}{\de T\inv x}_L = \mathcal{L}\left[\left(1-\tfrac{\de}{\ov{\lam_i}} \tfrac{\de}{\lam_j}\right) \ov{\xi_i}\xi_j\right]_{j,i=1}^n.
\]
Clearly $\left[\left(1-\tfrac{\de}{\ov{\lam_i}} \tfrac{\de}{\lam_j}\right) \ov{\xi_i}\xi_j\right]_{j,i=1}^n \in \mathcal{C}$ (take $a_{ij}=0, b_{ij}=\ov{\xi_i}\xi_j$ in the defining expression \eqref{Cone_C}), and so, since $\mathcal{L} \geq 0$ on $\mathcal{C}$, we have
\be\label{2stars}
\ip{x}{x}_L-\ip{\de T\inv x}{\de T\inv x}_L \geq 0.
\ee
Thus $ x \in {\mathcal N}$ implies that $ \de T\inv x \in {\mathcal N}$, and therefore
$\de T\inv$ induces an operator $\widetilde{(\de T\inv)}$ on  $\c^n /\mathcal{N}$ by 
\[
\widetilde{(\de T\inv)} (x+\mathcal{N}) = \de T\inv x + \mathcal{N},
\]
and in the light of inequality \eqref{2stars},
\be\label{3stars}
\|\widetilde{(\de T\inv)}\| \leq 1.
\ee
We have, for any $x\in\c^n$,
\[
\widetilde T \widetilde{(\de T\inv)} (x+\mathcal{N}) = \widetilde T (\de T\inv x+\mathcal{N}) = T \de T\inv x+\mathcal{N} = \de(x+\mathcal{N}),
\]
and so $\widetilde{(\de T\inv)} = \de (\widetilde T)\inv$.  Hence, by the inequality \eqref{3stars},
\[
\|\de (\widetilde T)\inv\| = \|\widetilde{(\de T\inv)} \| \leq 1.
\]

Therefore, $\tilde T$ is a Douglas-Paulsen operator.  Since the eigenvalues of $T$, which are $\lambda_1, \dots, \lambda_n$, belong to $R_\de$, $\sigma(T) \subseteq  R_\de$, and so the operator 
$T$ belongs to $\f_{\mathrm{dp}}(\de,\lam)$. 
Therefore, by Proposition \ref{T-g-G(lam)},  $ [ \langle\, e_j,e_i \,\rangle_L]_{i,j=1}^n $
belongs to $\mathcal{G}_{\mathrm {dp}} (\lambda)$.

Let $g_{ij}=\langle\, e_j,e_i \,\rangle_L$ for $i,j=1,\dots,n$.
By supposition \eqref{221},
$$ [(1-\overline z_i z_j)  \langle\, e_j,e_i \,\rangle_L]_{i,j=1}^n \ge 0.$$
Choose a polynomial $p$ such that $p(\lambda_i) = z_i$, $i =1,\dots, n$. Then
$~p(\widetilde T)e_i = z_i e_i$, $i =1,\dots, n$. Observe that
\begin{align*}
\left[\ip{(1-p(\tilde T)^*p(\tilde T))e_j}{e_i}\right] &= \left[ \ip{e_j}{e_i} - \ip{p(\tilde T)e_j}{p(\tilde T)e_i}\right] \\ 
	&= \left[ \ip{e_j}{e_i} - \ip{z_je_j}{z_ie_i}\right] \\ 
	&= \left[ (1- \overline z_i z_j)g_{ij} \right] \ge 0.
\end{align*}
 Therefore,
$\|p(\widetilde T)\| \le 1$. 
Choose $c = \begin{bmatrix} 1 \\ \dots \\ 1
\end{bmatrix} \in \c^n$. Then 
\[  \langle\,( 1- p(\widetilde T)^* p(\widetilde T)) c, c\,\rangle_L \ge 0,\]
that is, 
\[ {\mathcal L} ( [1-\overline z_i z_j]_{i,j=1}^n *c c^*) \ge 0, \ \text{and so } \ {\mathcal L} ( [1-\overline z_i z_j]_{i,j=1}^n ) \ge 0,
\]
where $*$ denotes the Schur product of matrices.

Thus, for every real linear functional ${\mathcal L}$ on  $H_n$ such that ${\mathcal L} \ge 0 $ on ${\mathcal C}$ we have $${\mathcal L} ( [1-\overline z_i z_j]_{i,j=1}^n ) \ge 0. $$
Hence $[1-\overline z_i z_j]_{i,j=1}^n$ belongs to ${\mathcal C}$. \end{proof}

We show in the next theorem that, as in the classical Pick theorem, if a DP Pick problem is solvable then it is solvable by a {\em rational} function in $\essdp$.

\begin{thm}  \label{extr-finite}
Let  $\lam_1,\dots,\lam_n \in R_\de$ be distinct  and let $z_1,\dots,z_n\in\c$.
If the DP Pick problem  
\[
\lambda_j \mapsto  z_j \quad \text{for} \ j=1,\ldots,n
\]
is solvable, 
 then there exists a rational function $\phi \in \essdp$ which satisfies the equations 
 \be\label{200ex-3}
\phi(\lambda_j) = z_j \qquad \text{for} \ j=1,\ldots,n,
\ee
and has a model $(\m,u)$, with $u:R_\de \to \m$ holomorphic, so that
\be\label{120-2}
1-\overline{\phi(\mu)}\phi(\lambda) = \Bigip{\big(1-E(\mu)^*E(\lambda)\big)\ u(\lambda)}{u(\mu)}_\m\ 
\text{for}\  \lambda,\mu \in R_\delta,
\ee
where $\m$ can be written as $\m=\m_1\oplus \m_2$, $\dim \m \leq 2n$ and $E:R_\delta \to \b(\m)$ is defined by the formula
\be\label{formE-4}
E(\lambda)=\begin{bmatrix}\lambda & 0\\ 0 & \frac{\delta}{\lambda}\end{bmatrix}, \ \qquad\text{for}\  \lambda \in R_\delta,
\ee
with respect to this orthogonal decomposition of $\m$.
\end{thm} 

\begin{proof}
Suppose that 
\[
\lam_j \mapsto z_j\ \text{for}\ j=1,\dots,n
\]
is a solvable DP-Pick problem.  By Theorem \ref{Pick-dp}, there exist positive semi-definite $n\times n$ matrices $a=\bbm a_{ij} \ebm$ and $b=\bbm b_{ij} \ebm$ such that 

\be\label{datamodel}
1-\overline z_i z_j = (1-\overline\lam_i\lam_j) a_{ij} + (1- \frac{\de^2}{\ov{\lam_i}\lam_j}) b_{ij} \quad \text{for} \ i,j=1,\dots,n.
\ee
Let the ranks of the matrices $a, b$ be $r_1,r_2$ respectively, so that $r_1\leq n, r_2\leq n$.  Then there exist vectors $x_1,\dots,x_n \in \c^{r_1}, y_1,\dots,y_n\in\c^{r_2}$ such that 
\[
a_{ij}= \ip{x_j}{x_i}_{\c^{r_1}}\ \text{and}\  b_{ij}= \ip{y_j}{y_i}_{\c^{r_2}} \quad \text{for}\  i,j=1,\dots,n.
\]
Substituting these relations into the equations \eqref{datamodel} and re-arranging, we obtain the relations
\[
1+\ip{\lam_j x_j}{\lam_i x_i}_{\c^{r_1}} +\ip{\frac{\de}{\lam_j}y_j}{\frac{\de}{\lam_i}y_i}_{\c^{r_2}} = \overline z_i z_j + \ip{x_j}{x_i}_{\c^{r_1}} +\ip {y_j}{y_i}_{\c^{r_2}}\ \text{for}\  i,j=1,\dots,n.
\]
These equations can in turn be expressed by saying that the families of vectors
\[
\bpm 1 \\ \lam_j x_j \\ \frac{\de}{\lam_j} y_j \epm_{j=1,\dots,n}\ \text{and}\ \bpm z_j \\ x_j \\ y_j \epm_{j=1,\dots,n}
\]
in $\c\oplus\c^{r_1}\oplus\c^{r_2}$ have the same gramians.  It follows from the ``lurking isometry lemma" \cite[Lemma 2.18]{amy20} that there exists an isometry $L\in \b(\c\oplus \c^{r_1}\oplus\c^{r_2})$ such that
\be\label{propL}
L\bpm 1 \\ \lam_j x_j \\ \frac{\de}{\lam_j} y_j \epm = \bpm z_j \\ x_j \\ y_j \epm\ \text{for}\ j=1,\dots,n.
\ee
Express $L$ by an operator matrix with respect to the orthogonal decomposition $\c\oplus\left(\c^{r_1} \oplus \c^{r_2}\right)$:
\[
L \sim \bbm a & 1\otimes \beta \\ \gamma\otimes 1 &D \ebm,
\]
where $a\in\c,\ \beta,\gamma \in \c^{r_1} \oplus \c^{r_2}$ and $D \in\b(\c^{r_1} \oplus \c^{r_2})$.  In terms of these variables and our previous notation
\[
E(\lam) \df \bbm \lam & 0\\ 0 & \frac{\de}{\lam} \ebm :\c^{r_1} \oplus \c^{r_2} \to \c^{r_1} \oplus \c^{r_2} \ \text{for}\ \lam\in R_\de,
\]
equation \eqref{propL} can be written
\begin{align} \label{propLbis}
a+ \ip{E(\lam_j)\bpm x_j \\ y_j\epm}{\beta}_{\c^{r_1} \oplus \c^{r_2}} &= z_j  \notag\\
\gamma + DE(\lam_j)\bpm x_j \\ y_j\epm &= \bpm x_j \\ y_j\epm
\end{align}
for $j=1,\dots,n$.  Observe that, for any $\lam\in R_\de$, $\|E(\lam)\| < 1$. As also $\|D\| \leq 1$ (since $L$ is an isometry), $1-DE(\lam_j)$ is invertible for each $j$.  The equations \eqref{propLbis} can therefore be solved to give
\begin{align}\label{solved}
\bpm x_j \\ y_j\epm &= (1-DE(\lam_j))\inv \gamma \notag \\
z_j &= a + \ip{E(\lam_j)(1-DE(\lam_j))\inv \gamma}{\beta}.
\end{align}
Now define $\ph\in\hol(R_\de)$ by 
\be\label{modelph}
\ph(\lam) = a + \ip{E(\lam)(1-DE(\lam))\inv \gamma}{\beta}_{\c^{r_1} \oplus \c^{r_2}}, \ \text{for} \ \lam \in R_\de.
\ee
By equation \eqref{solved}, $\ph(\lam_j) = z_j$ for $j=1,\dots,n$, and by \cite[Theorem 9.54]{amy20}, $\ph\in\ess_{\mathrm{dp}}$, while equation \eqref{modelph} constitutes a DP-realization for $\ph$.  By Cramer's rule for an invertible matrix, the function $\ph$ defined by equation \eqref{modelph} is a rational function.
 Accordingly, by Theorem \ref{thm3.8}, if we set $\m=\c^{r_1} \oplus \c^{r_2}$ and define a holomorphic function $u: R_\de \to \m$ by $u(\lam)=(1-E(\lam)D)\inv \gamma$, for $\lam \in R_\de$,
 then $(\m,u)$ as in equation \eqref{120-2} is a DP-model for $\ph$, while clearly $\dim \m =r_1+r_2 \leq 2n$. \end{proof}

\begin{rem}\label{PickandDPPick} {\em Solvable Pick data on $\d$ are also solvable as DP Pick data.}
Let  $\lambda_1,\ldots,\lambda_n$ be $n$ distinct points in $R_\delta$ and $z_1,\ldots,z_n \in \c$.
Suppose the Pick interpolation problem  on the open unit disc $\d$
\[
\lambda_j \mapsto  z_j \quad \text{for} \ j=1,\ldots,n
\]
is solvable. Then the  DP Pick Problem 
\[
\lambda_j \mapsto  z_j \quad \text{for} \ j=1,\ldots,n
\]
is also solvable. 
\end{rem}
\begin{proof}
By the assumption,
there exists a holomorphic function $\phi:\d \to \c$ such that $\phi(\lam_j)=z_j$ for $j=1,\dots,n$ and $\|\phi\|_{\hinf(\d)}\leq 1$. 
By Proposition \ref{phi-in-holD}, 
\be
\norm{\phi|R_\delta}_{\mathrm {dp}} =\|\phi\|_{\hinf(\d)}  \leq 1,
\ee
and so
the restriction of $\phi$ to $R_\delta$ is in $\mathcal{S}_{\mathrm {dp}}$,
which is to say that the corresponding DP Pick problem is solvable. \end{proof}

As the dp norm and sup norm are different, the converse statement to Remark \ref{PickandDPPick} is false, as one would expect.
The following two examples provide concrete instances of this fact.

\begin{ex} \label{DPsolv-notPsolv} {\em A solvable DP Pick data-set which is not a solvable Pick data-set on $\d$.}\\
Let $\delta \in (0,\half)$ and consider the $2$ distinct points  $\lambda_1 =\half, \lambda_2 = -\half$ in $R_\delta$. 
Recall that in Example \ref{dp-neq-infty} we showed that the function $\ph  \in \hol (R_\de)$, $\phi(\lambda) =  \half(\lam + \frac{\delta}{\lam})$ satisfies $\|\phi\|_{\mathrm{dp}} =1$.
Let $z_1=\phi(\lam_1)= \delta + \frac{1}{4}$, $z_2=\phi(\lam_2)= -(\delta + \frac{1}{4})$.
Thus the  DP Pick Problem 
\[
\lambda_j \mapsto  z_j \quad \text{for} \ j=1,2,
\]
is solvable by the function $\phi(\lambda) =  \half(\lambda + \frac{\delta}{\lambda})$. 

As to the Pick interpolation problem on the open unit disc $\d$
\[
\lambda_j \mapsto  z_j \quad \text{for} \ j=1,2,
\]
solvability depends on the value of $\de$. 
There are $3$ cases:
\begin{enumerate}[(i)]
\item for $\delta \in (0,\frac{1}{4})$,
the Pick interpolation problem
\[
\lambda_j \mapsto  z_j \quad \text{for} \ j=1,2,
\]
is solvable;
\item for $\delta = \frac{1}{4}$,
the Pick interpolation problem
\[
\lambda_j \mapsto  z_j \quad \text{for} \ j=1,2,
\]
 on the open unit disc $\d$ is extremally solvable and has the unique solution $f(\lambda) = \lambda$;
\item for $\delta \in (\frac{1}{4},\frac{1}{2} )$,
the Pick interpolation problem
\[
\lambda_j \mapsto  z_j \quad \text{for} \ j=1,2,
\]
is not solvable on $\d$.
\end{enumerate}
\end{ex}
\begin{proof} To prove (i)-(iii) on the solvability of the Pick interpolation problem on the open unit disc $\d$
\[
\lambda_j \mapsto  z_j \quad \text{for} \ j=1,2,
\]
we  consider the appropriate Pick matrix, which here is
\[
P(\delta) = \displaystyle \bbm \frac{1-\ov {z_i} z_j}{1-\ov{\lam_i}\lam_j} \ebm_{i,j=1}^2. 
\]
That is,
\be
P(\delta) =\begin{bmatrix} \frac{1-(\delta + \frac{1}{4})^2}{ 1- \frac{1}{4}} &
 \frac{1+(\delta + \frac{1}{4})^2}{ 1+ \frac{1}{4}} \\ \\
 \frac{1+(\delta + \frac{1}{4})^2}{ 1+ \frac{1}{4}} & \frac{1-(\delta + \frac{1}{4})^2}{ 1- \frac{1}{4}}.\\
\end{bmatrix}
\ee
It is clear that, for  $\delta \in (0,\half)$,
\[
P(\delta)_{11} =\frac{1-(\delta + \frac{1}{4})^2}{ 1- \frac{1}{4}} > 0.
\]
A little calculation shows that the determinant of the Pick matrix
\[
\det P(\de)= \frac{16^2}{15^2}\left\{\de^2+\half\de-\frac{3}{16}\right\}\left\{\de^2+\half\de-\frac{63}{16}\right\},
\]
from which one can deduce that 
\begin{enumerate}[(i)]
\item
when  $\delta \in (0,\frac{1}{4})$, $\det P(\delta) >0$ and so $P(\de)> 0$; 
\item $\det P(\delta) =0$, when $\delta = \frac{1}{4}$;
and 
\item $\det P(\delta) <0$, when $\delta \in (\frac{1}{4},\frac{1}{2} )$.
\end{enumerate}
Therefore the  Pick matrix $P(\delta)$ is not positive when $\delta \in (\frac{1}{4},\frac{1}{2} )$. Thus, by Pick's theorem, for $\delta \in (\frac{1}{4},\frac{1}{2} )$,
the Pick interpolation problem
\[
\lambda_j \mapsto  z_j \quad \text{for} \ j=1,2,
\]
is not solvable, while, for $\de=\tfrac 14$, the Pick interpolation problem is uniquely solvable, and one sees by inspection that the unique solution is the function $f(\lam)=\lam$.
\end{proof}

\begin{ex} \label{notPicksolvable}  {\em Another solvable DP Pick data-set which is not a solvable Pick data-set on $\d$.} 
Let $\de\in (0,1)$ and let $\lam_1 = \frac{\de+\sqrt{\de}}{2}$, $\lam_2=-\lam_1$.  
We have $0 <\de < \frac{\de+\sqrt{\de}}{2} <\sqrt{\de} <1$, so that $\lam_1,\lam_2 \in R_\de$.
Recall from Example \ref{dp-neq-infty} that the function $\phi(z)=\tfrac 12(z+\frac{\de}{z})$ on $R_\de$ satisfies $\|\phi\|_{\mathrm{dp}} =1$.
Consider the DP-Pick problem 
\be\label{nPs}
\lam_i \mapsto z_i \df  \phi(\lam_i), i=1,2.
\ee
  Clearly this is a solvable DP-Pick problem, with solution $\ph$.  However, the Pick problem with the same data $\lam_j \mapsto z_j, j=1,2$, is not solvable.   Indeed, the Pick matrix for the problem \eqref{nPs} is
\[
P= \bbm \frac{1-|z_1|^2}{1-|\lam_1|^2} &  \frac{1+|z_1|^2}{1+|\lam_1|^2} \\
 \frac{1+|z_1|^2}{1+|\lam_1|^2} & \frac{1-|z_1|^2}{1-|\lam_1|^2} \ebm.
 \]
Thus
\begin{align*}
\det P &= \left(\frac{1-|z_1|^2}{1-|\lam_1|^2}\right)^2 - \left(\frac{1+|z_1|^2}{1+|\lam_1|^2}\right)^2 \\
	&= D_1D_2
\end{align*}
where
\begin{align*}
D_1 &= \frac{1-|z_1|^2}{1-|\lam_1|^2} - \frac{1+|z_1|^2}{1+|\lam_1|^2} \\
	&= \frac{2(|\lam_1|^2 - |z_1|^2)}{1-|\lam_1|^4}, \\
D_2 &= \frac{1-|z_1|^2}{1-|\lam_1|^2} + \frac{1+|z_1|^2}{1+|\lam_1|^2} \\
	&=\frac{2(1-|z_1\lam_1|^2)}{1-|\lam_1|^4}.
\end{align*}
Now
\begin{align*}
z_1 &= \frac {1}{2} \left(\lam_1+\frac{\de}{\lam_1} \right)= \frac {1}{2}\left(\frac{\de+\sqrt{\de}}{2} + \frac{2\de}{\de+\sqrt{\de}}\right) \\
	&= \frac{\sqrt{\de}}{2}\left(\frac{1+\sqrt{\de}}{2} + \frac{2}{1+\sqrt{\de}}\right) 
	= \frac{\sqrt{\de}(5+2\sqrt{\de} +\de)}{4(1+\sqrt{\de})},
\end{align*}
and
\begin{align*}
0 < z_1\lam_1 &= \frac{\sqrt{\de}(5+2\sqrt{\de} +\de)}{4(1+\sqrt{\de})}. \frac{\sqrt{\de}(1+\sqrt{\de})}{2}
	= \frac{\de(5+2\sqrt{\de} +\de)}{8} \\
		&< 1.
\end{align*}
Thus $D_2 > 0$, and moreover
\begin{align*}
|\lam_1| - |z_1| &= \lam_1- \half (\lam_1 + \frac{\de}{\lam_1}) = \half\left(\frac{\de+\sqrt{\de}}{2} - \frac{2\de}{\de+\sqrt{\de}} \right) \\
	&= \frac{\sqrt{\de}}{2}\left( \frac{1+\sqrt{\de}}{2} - \frac{2}{1+\sqrt{\de}}\right) \\
	&= \frac{\sqrt{\de}(-3+2\sqrt{\de}+\de)}{4(1+\sqrt{\de})} \\
	&< 0,
\end{align*}
from which it follows that $D_1 < 0$, and hence $D < 0$. Thus the Pick matrix $P$ is not positive, and so, by Pick's Theorem, the Pick interpolation problem $\lam_j \mapsto z_j, j=1,2$, is not solvable.
\end{ex}

Since the Pick interpolation problem on $\d$ and the DP Pick problem on $R_\de$ are so closely related, it is natural to ask whether the Szeg\H{o} kernel on $\d$, when retricted to $R_\de$, is a DP-Szeg\H{o} kernel.   We can use Example \ref{notPicksolvable} to answer this question in the negative.

\begin{prop}\label{Szego-not-DP} Let $\delta \in (0,1)$.
The Szeg\H{o} kernel $[\frac{1}{1-\overline\mu \lambda}]$ restricted to $R_\delta$ is not a DP Szeg\H{o} kernel on $R_\delta$.
\end{prop}
\begin{proof}
Suppose the kernel $[\frac{1}{1-\overline\mu \lambda}]$ restricted to $R_\delta$ is a DP kernel. Then, for any distinct $\lam_1,\dots,\lam_n \in R_\de$,
the localization of  $[\frac{1}{1-\overline\mu \lambda}]$ to $\{\lam_1,\dots,\lam_n \}$
belongs to $\mathcal{G}_{\mathrm {dp}} (\lambda)$.

Consider the $2$ distinct points  
 $\lam_1 = \frac{\de+\sqrt{\de}}{2}$, $\lam_2=-\lam_1$, note that, for $\delta \in (0,1)$, $\lam_1,\lam_2 \in R_\de$.
By Example \ref{dp-neq-infty}, the function $\phi(z)=\tfrac 12(z+\frac{\de}{z})$ on $R_\de$ satisfies $\|\phi\|_{\mathrm{dp}} =1$. Therefore,
for $\lambda_i$ and   $z_i= \phi(\lambda_i)= \frac {1}{2} \left(\lam_i+\frac{\de}{\lam_i} \right)$,
$ i=1,2$, the  DP Pick problem
\be \label{pick-ex-1}
\lambda_j \mapsto  z_j \quad \text{for} \ j=1,2,
\ee
is solvable. In Example \ref{notPicksolvable} we showed that, for $\delta \in (0,1)$, the corresponding Pick problem \eqref{pick-ex-1} is not solvable.

Since the problem \eqref{pick-ex-1} is a solvable DP Pick problem, 
by Theorem \ref{iff-dp-kernel-z}, for all 
$g \in \mathcal{G}_{\mathrm {dp}} (\lambda)$,
\be\label{212-5}
[(1-\overline z_i z_j)g_{ij}] \ge 0.
\ee
By the assumption, the localization of  $[\frac{1}{1-\overline\mu \lambda}]$ to $\{\lam_1,\lam_2 \}$ belongs to $\mathcal{G}_{\mathrm {dp}} (\lambda)$.
In particular, for the Pick problem
\be \label{pick-ex}
\lambda_j \mapsto  z_j \quad \text{for} \ j=1,2,
\ee
on $\d$, the Pick matrix 
\be\label{Pick-22}
\left[\frac{1-\overline z_i z_j}{1-\overline\lambda_i\lambda_j}\right]_{i,j=1}^2 \ge 0.
\ee
Hence, by Pick's theorem, the problem \eqref{pick-ex} is solvable by a Schur function $f$ on $\d$.
This contradicts our example, and so $[\frac{1}{1-\overline\mu \lambda}]$ is not a DP Szeg\H{o} kernel on $R_\delta$. \end{proof}

\section{Extremal DP Pick problems }\label{extremal}

In this section we study DP Pick interpolation problems that are ``only just" solvable.
 We say that a DP Pick problem is \emph{extremally solvable} if it is solvable and there does not exist $\phi \in \hinf_{\rm dp}$ with $\norm{\phi}_{\rm dp} <1$ satisfying the equations 
\be\label{200ex}
\phi(\lambda_j) = z_j \qquad \text{for} \ j=1,\ldots,n.
\ee

\begin{rem} \label{non-unique}
A DP Pick problem that is not extremally solvable cannot have a unique solution.  For suppose $\lam_j \mapsto z_j, j=1,\dots,n,$ is a solvable DP Pick problem that is not extremally solvable.  That means that there is a function $\ph:R_\de \to \c$ such that $\ph(\lam_j)=z_j$ for $j=1, \dots,n$ and $\|\ph\|_{\mathrm{dp}}<1$.
Consider the function $\psi(\lam)=\ph(\lam) + \eps \prod_{j=1}^n (\lam-\lam_j)$, for $\lam\in R_\de$, for some positive $\eps$.  Then $\psi(\lam_j)=z_j$ for $j=1, \dots,n$ and 
$$\|\psi\|_{\mathrm{dp}}\leq \|\ph\|_{\mathrm{dp}} + \eps\|\prod_{j=1}^n (\lam-\lam_j)\|_{\mathrm{dp}}<1$$ for all small enough $\eps$, and so there are infinitely many solutions to the interpolation problem $\lam_j \mapsto z_j$ for $j=1,\dots,n$ having DP norm less than $1$.
\end{rem}

Next we give necessary and sufficient conditions for a DP Pick problem to be extremally solvable.
\begin{thm} \label{DP-extremal}
Let  $\lam_1,\dots,\lam_n \in R_\de$ be distinct  and let $z_1,\dots,z_n\in\c$.
The following two statements are equivalent.
\begin{enumerate}[(i)]
\item The DP Pick problem  
\[
\lambda_j \mapsto  z_j \quad \text{for} \ j=1,\ldots,n
\]
is extremally solvable.
\item  For all  $ g \in \mathcal{G}_{\mathrm {dp}} (\lambda)$,
\be\label{215}
[(1-\overline z_i z_j)g_{ij}]_{i,j=1}^n \ge 0
\ee
and there exists $ \widetilde{g}\in \gdp(\lambda)$ such that
\be\label{500}
  \rank \big[(1-\overline z_i z_j)\widetilde{g}_{ij}\big]_{i,j=1}^n  < n.
\ee
\end{enumerate}
\end{thm}

\begin{proof} (i) $\implies$ (ii).
Suppose that the DP Pick problem 
\[
\lam_j \mapsto z_j\ \text{for}\  j=1,\dots,n
\]
 is extremally solvable.  Since the problem is solvable, Theorem \ref{iff-dp-kernel-z} implies that,
 for all  $g \in \mathcal{G}_{\mathrm {dp}} (\lambda)$,
\be\label{212-2}
[(1-\overline z_i z_j)g_{ij}] \ge 0.
\ee

Suppose, for a contradiction, that there is no $g \in \mathcal{G}_{\mathrm {dp}} (\lambda)$ such that
\be\label{212-3}
[(1-\overline z_i z_j)g_{ij}]\ \text{is singular}.
\ee
Let $F:\r\times \mathcal{G}_{\mathrm {dp}}^{\mathrm{norm}} (\lambda) \to \r$ be defined by 
\begin{align*}
F(r,[g_{ij}]) &= \text{the minimum of the leading principal minors of} \bbm (1-r^2\overline z_i z_j) g_{ij}\ebm_{i,j=1}^n \notag\\
 &=\min_{J=1,\dots,n} \det \bbm (1-r^2\overline z_i z_j) g_{ij}\ebm_{i,j=1}^J.
\end{align*}
By standard linear algebra, for any positive definite matrix $g=[g_{ij}]$, $F(r,g)> 0$ if and only if 
$\bbm (1-r^2\overline z_i z_j) g_{ij}\ebm_{i,j=1}^n > 0$.
  $F$ is continuous and, by supposition,  $$ [(1-\overline z_i z_j)g_{ij}] > 0$$
 for all  $g \in \mathcal{G}_{\mathrm {dp}} (\lambda)$,
  which implies that $F(1,g) > 0$ for all $g \in \mathcal{G}_{\mathrm {dp}}^{\mathrm{norm}} (\lambda)$.  Since, by Proposition \ref{dpcompact}, $\mathcal{G}_{\mathrm {dp}}^{\mathrm{norm}} (\lambda)$ is compact, $F(1,\cdot)$ attains its minimum on $\mathcal{G}_{\mathrm {dp}}^{\mathrm{norm}} (\lambda)$, and so there exists $\kappa>0$ such that $F(1,g)\geq\kappa$ for all $g\in \mathcal{G}_{\mathrm {dp}}^{\mathrm{norm}} (\lambda)$.
  
  By the continuity of $F$ and, again by the compactness of $\mathcal{G}_{\mathrm {dp}}^{\mathrm{norm}}(\lambda)$, the family of functions $\{ F(\cdot,g):g \in \mathcal{G}_{\mathrm {dp}}^{\mathrm{norm}} (\lambda)\}$ is equicontinuous on $\r$.  Hence there exists $\de> 0$ such that $F(r,g) > 0$ for all $g\in \mathcal{G}_{\mathrm {dp}}^{\mathrm{norm}} (\lambda)$ and all $r\in (1,\de)$. Choose some $r\in (1,\de)$.  Then $F(r,g) > 0$ for all $g\in \mathcal{G}_{\mathrm {dp}}^{\mathrm{norm}} (\lambda)$, and therefore $\bbm (1-r^2\overline z_i z_j) g_{ij}\ebm_{i,j=1}^n > 0$ for all $g\in \mathcal{G}_{\mathrm {dp}}^{\mathrm{norm}} (\lambda)$.  It follows from Proposition \ref{dpcompact} that  $\bbm (1-r^2\overline z_i z_j) g_{ij}\ebm_{i,j=1}^n > 0$ for all $g\in\mathcal{G}_{\mathrm {dp}} (\lambda)$.  Hence, by Theorem \ref{iff-dp-kernel-z}, for the chosen   $r\in (1,\de)$, the DP Pick problem
\[
\lam_j \mapsto rz_j\ \text{for}\ j=1,\dots,n
\]
is solvable, which is to say that there exists a function $\psi\in\hol(R_\de)$ such that $\|\psi\|_{\mathrm{dp}} \leq 1$ and $\psi(\lam_j)=rz_j$ for $j=1,\dots,n$.  Thus the function $\ph\df\psi/r$ satisfies $\ph(\lam_j)=z_j$ for $j=1,\dots,n$ and   $\|\ph\|_{\mathrm{dp}} \leq 1/r < 1$, contrary to hypothesis.  Hence there is a $g \in \mathcal{G}_{\mathrm {dp}} (\lambda)$ such that
\be\label{212-4}
[(1-\overline z_i z_j)g_{ij}]\ \text{is singular}.
\ee
 We have shown that statements \eqref{215} and \eqref{500} hold, and so have established  (i) $\Rightarrow$ (ii) necessity in Theorem \ref{DP-extremal}.\\

 (ii) $\Rightarrow$ (i).
Suppose that (ii) holds, and so, for all  $ g \in \mathcal{G}_{\mathrm {dp}} (\lambda)$,
\be\label{215-2}
[(1-\overline z_i z_j)g_{ij}]_{i,j=1}^n \ge 0.
\ee
By Theorem \ref{iff-dp-kernel-z}, 
there exists $ \phi \in \essdp$ such that
\be \label{inter-extr}
\phi(\lambda_j) = z_j \qquad \text{for} \ j=1,\ldots,n. 
\ee

Suppose (i) does not hold, which means that the problem is non-extremally solvable, and hence there exists $\ph$ such that $\|\ph\|_{\mathrm{dp}} = r < 1$ and $\ph$ satisfies $\ph(\lam_j)=z_j$ for $j=1,\dots,n$.
Thus  for all  $ g \in \mathcal{G}_{\mathrm {dp}} (\lambda)$,
\be\label{r10}
[(r^2-\overline{z_i}z_j)g_{ij}]\ge 0.
\ee
By assumption (ii),  there exists $ \widetilde{g}\in \gdp(\lambda)$ such that
\be\label{500-2}
  \rank \big[(1-\overline z_i z_j)\widetilde{g}_{ij}\big]_{i,j=1}^n  < n.
\ee
and hence $[(1-\overline z_i z_j)\widetilde{g}_{ij}]_{i,j=1}^n$ has a non-zero null vector $v$.
Consider the relation
\[
(1-\overline z_i z_j) \widetilde g_{ij} = (1-r^2) \widetilde g_{ij} +(r^2-\overline z_i z_j) \widetilde g_{ij}.
\]
Since  $\widetilde g_{ij}>0$,  $(1-r^2) \widetilde g_{ij}>0$ and, by equation \eqref{r10}, $(r^2-\overline z_i z_j) \widetilde g_{ij} \geq 0$,
which is a contradiction. \end{proof}

By Theorem \ref{extr-finite}, if a DP Pick problem is solvable, then there exists
 a rational solution $\phi \in \essdp$.  In the next theorem we show that if, further, the problem is {\em extremally} solvable then there exists $T\in \fdp(\de,\lambda)$, acting on an $n$-dimensional Hilbert space, such that
$\norm{\phi(T)}=\norm{\phi}_{\rm dp} = 1.$

\begin{thm}  \label{finite-extr}
Let  $\lam_1,\dots,\lam_n \in R_\de$ be distinct  and let $z_1,\dots,z_n\in\c$.
If the DP Pick problem  
\[
\lambda_j \mapsto  z_j \quad \text{for} \ j=1,\ldots,n,
\]
is extremally solvable,  then there exists a rational function $\phi \in \essdp$ which satisfies the equations 
 \be\label{200ex-4}
\phi(\lambda_j) = z_j \qquad \text{for} \ j=1,\ldots,n,
\ee
and has a model $(\m,u)$  as in equation \eqref{120-2}, where  $u: R_\de \to \m$ is a holomorphic function and $\dim \m \leq 2n$. Furthermore, there exists $T\in \fdp(\de,\lambda)$ such that
\[
1=\norm{\phi}_{\rm dp} = \norm{\phi(T)}.
\]
In particular,
\[
1-\phi(T)^*\phi(T) = u(T)^*\big(1-E(T)^*E(T)\big)\ u(T),
\]
where $\m$ can be written as $\m=\m_1\oplus \m_2$ and $E:R_\delta \to \b(\m)$ is defined by the formula
\be\label{formE-2}
E(\lambda)=\begin{bmatrix}\lambda & 0\\ 0 & \frac{\delta}{\lambda}\end{bmatrix}\ \qquad\text{for}\  \lambda \in R_\delta
\ee
with respect to this orthogonal decomposition of $\m$.
\end{thm} 
\begin{proof}  Since the DP Pick problem $\lambda_j \mapsto  z_j, \ \text{for} \ j=1,\ldots,n,$ is solvable, by Theorem \ref{extr-finite}, there exists
 a rational function $\phi \in \essdp$ such that $\phi(\lambda_j) = z_j, \text{for} \ j=1,\ldots,n$. 
 
Let us now prove the existence of an operator $T$ with the stated properties. By assumption, the DP-Pick problem $\lam_j \mapsto z_j, j=1,\dots,n$ is {\em extremally} solvable. 
By Theorem \ref{DP-extremal}, there exists $\widetilde g \in \gdp(\lambda)$ such that
\be\label{500bis}
  \rank \big[(1-\overline z_i z_j)\widetilde{g}_{ij}\big]_{i,j=1}^n  < n,
\ee
so that $\bbm (1-\overline z_i z_j)\widetilde{g}_{ij} \ebm$ is singular, and therefore has a non-zero null vector $\xi=[\xi_1,\dots,\xi_n]^T \in\c^n$, which is to say that
\be\label{propxi}
\sum_{j=1}^n (1-\overline z_i z_j)\widetilde{g}_{ij} \xi_j = 0\ \text{for}\ i=1,\dots,n.
\ee
Since $\widetilde g \in  \gdp(\lambda)$, $[\widetilde g_{ij}] > 0$, and so $[\widetilde g_{ij}]$ has rank $n$. By Moore’s theorem's Theorem there exist an $n$-dimensional Hilbert space $\h$ and a basis $\widetilde e_1, \dots,\widetilde e_n \in \h$ such that $\widetilde g_{ij}=\ip{\widetilde e_j}{\widetilde e_i}$ for $i,j=1,\dots,n$.

Define an operator $T$ on $\h$ by $T\widetilde e_j = \lam_j\widetilde e_j$ for $j=1,\dots,n$.  Since $\widetilde g \in \gdp(\lambda)$, by Proposition \ref{g-to-T-DP},  $T\in \fdp(\delta,\lambda)$.  Note that $\ph(T)\widetilde e_j=z_j\widetilde e_j$, and so, if $x=\sum_{j=1}^n \xi_j\widetilde e_j$, then
\[
\ph(T)x = \sum_{j=1}^n z_j\xi_j\widetilde e_j
\]
\noindent and
\begin{align}\label{phTcontr}
\ip{(1-\ph(T)^*\ph(T))x}{x} &= \sum_{i,j=1}^n (1-\overline z_iz_j) \overline\xi_i \xi_j\ip{\widetilde e_j}{\widetilde e_i}\notag \\ 	  &=\sum_{i,j=1}^n (1-\overline z_iz_j) \overline\xi_i \xi_j\widetilde g_{ij} \notag\\
	&= \sum_{i=1}^n \overline \xi_i\sum_{j=1}^n (1-\overline z_iz_j) \xi_j\widetilde g_{ij} \notag\\
	&= 0.
\end{align}
As $\xi\neq 0$, the complex numbers $\xi_1,\dots,\xi_n$ are not all zero, and so, since $\widetilde e_1,\dots,\widetilde e_n$ are linearly independent, $x=\sum_{j=1}^n \xi_j\widetilde e_j \neq 0$.  Since  $\|\ph\|_{\mathrm{dp}} \leq 1$
and $T\in \fdp(\delta,\lambda)$, we have $\|\ph(T)\|\leq 1$, and so $1-\ph(T)^*\ph(T) \geq 0$.  In conjunction with the equality \eqref{phTcontr}, this implies that $(1-\ph(T)^*\ph(T))x=0$, and hence $\|\ph(T)x\|^2=\|x\|^2$.  Since $x\neq 0$, $x$ is a maximizing vector for $\ph(T)$ and $\|\ph(T)\|=1$.

By Theorem \ref{extr-finite}, for the rational function $\phi$ there exists a model $(\m, u),$ where  $u: R_\de \to \m$ is holomorphic, so that
\be\label{120-5}
1-\overline{\phi(\mu)}\phi(\lambda) = \Bigip{\big(1-E(\mu)^*E(\lambda)\big)\ u(\lambda)}{u(\mu)}_\m\ 
\text{for}\  \lambda,\mu \in R_\delta,
\ee
where  $\dim \m \leq 2n$.
Since $T\in \fdp(\delta,\lambda)$, $T$ 
satisfies
\[
\sigma(T)=\{\lambda_1,\ldots,\lambda_n\} \subset R_\delta.
\]
Thus, by the Riesz-Dunford functional calculus, $\ph(T)$ is well defined and, by the hereditary functional calculus,
\[
1-\phi(T)^*\phi(T) = u(T)^*\big(1-E(T)^*E(T)\big)\ u(T).
\qedhere
\]
\end{proof}

\section{Declarations}
{\bf NSF and EPSRC grants.}
Agler was partially supported by National Science Foundation Grants DMS 1361720 and 1665260.
Lykova and Young were partially supported by the Engineering and Physical Sciences grant EP/N03242X/1.\\

{\bf Conﬂict of interest.} The authors have no Conﬂict of interest to declare that are
relevant to the content of this article.

{\bf Data availability statement.} 
 No data were collected, generated or consulted in connection with this research.\\

\end{document}